\documentclass[11pt]{article}
\usepackage{amsthm}
\usepackage{mathrsfs}

\makeatletter
\def\th@plain{%
  \thm@notefont{}
  \itshape 
}
\def\th@definition{%
  \thm@notefont{}
  \normalfont 
}
\makeatother
\usepackage{amsthm}
\usepackage{mathrsfs}
\usepackage{amsmath}
\usepackage{amsfonts}
\usepackage{latexsym}
\usepackage{amssymb}
\usepackage{graphicx}
\usepackage{tikz}
\usepackage{srcltx,graphicx}
\theoremstyle{plain}
\usepackage{multirow}
\usepackage{soul}
\usepackage{tabularx,ragged2e}
\newcolumntype{C}{>{\Centering\arraybackslash}X} 
\usepackage{caption}
\usepackage{subcaption}
\usepackage{float}
\usepackage{xcolor}
\usepackage{pdfpages}

\newtheorem*{theorem*}{Theorem}
\newtheorem{theorem}{Theorem}[section]
\newtheorem{informaltheorem}[theorem]{Informal Theorem}
\newtheorem{assumption}[theorem]{Assumption}
\newtheorem{corollary}[theorem]{Corollary}
\newtheorem{definition}[theorem]{Definition}

\newtheorem{example}[theorem]{Example}
\newtheorem{lemma}[theorem]{Lemma}
\newtheorem*{lemma*}{Lemma}
\newtheorem*{definition*}{Definition}

\newtheorem{remark}[theorem]{Remark}

\numberwithin{equation}{section}
\numberwithin{figure}{section}
\usepackage[a4paper]{geometry}
\newcommand{\beq}{\begin{equation}}
\newcommand{\eeq}{\end{equation}}

\newcommand{\Ccont}{C_{\mathrm{cont}}}

\newcommand{\Cpou}{C_{\mathrm{PoU}}}

\newcommand{\Hsub}{H_{\mathrm{sub}}}

\newcommand{\cJ}{{\mathcal J}}

\newcommand{\cP}{{\mathcal P}}

\newcommand{\cS}{{\mathcal S}}
\newcommand{\cT}{{\mathcal T}}

\newcommand{\cV}{{\mathcal V}}

\newcommand{\bx}{\mathbf{x}}

\newcommand{\bV}{\mathbf{V}}

\newcommand{\bW}{\mathbf{W}}

\newcommand{\bY}{\mathbf{Y}}

\newcommand{\supp}{\mathrm{supp}}

\newcommand{\bZ}{\mathbf{Z}}

\newcommand{\Rea}{\mathbb{R}}
\newcommand{\Com}{\mathbb{C}}

\newcommand{\ri}{{\rm i}}

\usepackage{color}
\usepackage{hyperref}
\definecolor{myblue}{rgb}{0,0,0.6}
\definecolor{darkgreen}{rgb}{0,0.5,0}
\hypersetup{colorlinks=true,
linkcolor=myblue,citecolor=myblue,filecolor=myblue,urlcolor=myblue}

\definecolor{escol}{rgb}{0,0,0.6}
\definecolor{sgcol}{rgb}{0,0,0.7}
\definecolor{estcol}{rgb}{0.5,0,0}
\definecolor{esnewcol}{rgb}{0,0.5,0}
\definecolor{lightgrayl}{RGB}{198,198,198}
\newcommand{\red}[1]{{\color{black}{#1}}}

\newcommand{\beqs}{\begin{equation*}}
\newcommand{\eeqs}{\end{equation*}}
\newcommand{\bit}{\begin{itemize}}
\newcommand{\eit}{\end{itemize}}
\newcommand{\ben}{\begin{enumerate}}
\newcommand{\een}{\end{enumerate}}
\newcommand{\bal}{\begin{align}}
\newcommand{\eal}{\end{align}}
\newcommand{\bals}{\begin{align*}}
\newcommand{\eals}{\end{align*}}
\newcommand{\bre}{\begin{remark}}
\newcommand{\ere}{\end{remark}}
\newcommand{\bpf}{\begin{proof}}
\newcommand{\epf}{\end{proof}}
\newcommand{\ble}{\begin{lemma}}
\newcommand{\ele}{\end{lemma}}
\newcommand{\bco}{\begin{corollary}}
\newcommand{\eco}{\end{corollary}}
\newcommand{\bex}{\begin{example}}
\newcommand{\eex}{\end{example}}
\newcommand{\bth}{\begin{theorem}}
\newcommand{\enth}{\end{theorem}}

\newcommand{\tfa}{\text{ for all }}

\newcommand{\tin}{\text{ in }}

\newcommand{\tas}{\text{ as }}
\newcommand{\tand}{\text{ and }}

\newcommand{\Csol}{C_{\rm sol}}
\newcommand{\Cgo}{}
\newcommand{\Cshape}{C_{\rm sr}}

\newcommand{\Cnormo}{C_{Q_0}}
\newcommand{\Cnormell}{C_{{\rm sub}}}
\newcommand{\Cdegree}{C_{\rm deg, f}}
\newcommand{\Cdegreet}{C_{\rm deg, f}}
\newcommand{\Cdegreec}{C_{\rm deg, c}}

\newcommand{\Ccoarse}{C_{\rm coa}}
\newcommand{\Hcoarse}{{H_{\rm coa}}}

\newcommand*{\N}[1]{\left\|#1\right\|}

\newcommand{\tendi}{\rightarrow\infty}

\newcommand{\mythmname}[1]{\textbf{\emph{(#1).}}}

\newcommand{\vertiii}[1]{{\left\vert\kern-0.25ex\left\vert\kern-0.25ex\left\vert #1 
    \right\vert\kern-0.25ex\right\vert\kern-0.25ex\right\vert}}

\usepackage[outdir=./]{epstopdf}
\usepackage{verbatim}

\usepackage{enumitem}

\newcommand{\fine}{\cV_h}
\newcommand{\coarse}{\cV_0}

\newcommand{\subdomain}{\cV_{\ell}}

\newcommand{\matrixD}{\mathcal{D}}
\newcommand{\matrixA}{\mathcal{A}}

\newcommand{\matrixAzero}{\matrixA_{0}}

\newcommand{\matrixR}{\mathcal{ R}}

\newcommand{\matrixB}{\mathcal{ B}}

\newcommand{\fdspace}{\fine}

\newcommand{\polyexp}{P}
\newcommand{\opP}{\cP}

\newcommand{\inter}{\mathfrak{J}}
\newcommand{\Qa}{Q}

\newcommand{\matrixBadd}{\matrixB}

\title{
Theory of two-level Schwarz preconditioners with piecewise-polynomial coarse spaces for the high-frequency Helmholtz equation
}

\author{
 I.~G.~Graham\thanks{Department of Mathematical Sciences, University of Bath, UK, {\tt i.g.graham@bath.ac.uk}} 
\,\,\,and\,
E.~A.~Spence\thanks{Department of Mathematical Sciences, University of Bath, UK, {\tt e.a.spence@bath.ac.uk}}
}

\date{\today}

\begin{document}
\maketitle

\begin{abstract}
We analyse the classic two-level additive Schwarz domain-decomposition GMRES preconditioner for finite-element discretisations of the Helmholtz equation with \red{large} wavenumber $k$, where 
both the fine and coarse spaces consist of piecewise polynomials with polynomial degree increasing like $\log k$.

We exhibit choices of these fine and coarse spaces such that -- up to factors of $\log k$ -- 
both are pollution free (with the ratio of the coarse-space dimension to the fine-space dimension arbitrarily small),  
the number of degrees of freedom per subdomain 
is constant, and  the number of GMRES iterations is \red{proved to be} bounded independently of $k$.

These are the first \red{$k$-explicit} convergence results about a two-level Schwarz preconditioner 
for high-frequency Helmholtz 
with a coarse space that is pollution free and does \emph{not} consist of problem-adapted basis functions.
\end{abstract}


\section{Introduction}

\subsection{Context and motivation:~\red{coarse spaces for high-frequency Helmholtz}}\label{sec:context}

When solving self-adjoint positive-definite problems (such as Laplace's equation) with domain-decomposition (DD) methods, coarse spaces provide global transfer of information, and are the key to parallel scalability (see, e.g., \cite{ToWi:05, Wi:09}, \cite[Chapter 4]{DoJoNa:15}). However, the design of practical coarse spaces for high-frequency wave problems, such as the high-frequency Helmholtz equation, is much more difficult than in the self-adjoint positive-definite case (see, e.g., the recent computational study \cite{BoDoJoTo:21} and the references therein) and there have been several papers appearing in the last year \red{or so} on the rigorous numerical analysis of this question 
\cite{HuLi:24, NaPa:24, LuXuZhZo:24, MaAlSc:24, FuGoLiWa:24}. 

The heart of the issue is  that the accurate approximation of a function oscillating at frequency $\lesssim k$ in a domain of characteristic length scale $L$ requires $\sim (kL)^d$ degrees of freedom \red{(where $d$ is the spatial dimension)}. Furthermore, the \emph{pollution effect} \cite{BaSa:00} means that finite-element methods (FEMs) with fixed polynomial degree applied to the Helmholtz equation require  $\gg(kL)^d$ degrees of freedom to be accurate. In particular, \cite{BeSaToVe:23} recently exhibited examples of meshes in 2-d with $(kL)^d$ degrees of freedom  for which the Helmholtz FEM solution with fixed polynomial degree does not exist. In contrast, if the polynomial degree increases logarithmically with $\log (kL)$, then the resulting $hp$-FEM method does \emph{not} suffer from the pollution effect \cite{MeSa:10, MeSa:11, EsMe:12, LSW3, LSW4, GLSW1, BeChMe:24}, assuming that the solution operator of the Helmholtz problem grows polynomially with $kL$ (which  \cite{LSW1} showed is true for ``most" frequencies).

A popular strategy for designing coarse spaces is to solve appropriate local problems to create basis functions adapted to the underlying Helmholtz problem, and then glue these together using a partition of unity.  
Two such spaces specifically designed as DD coarse spaces are the GenEO (generalized eigenvalue problems on the overlap) method 
\cite{SpDoHaNaPeSc:14, BoDoGrMaSc:23} and the method of \cite{CDKN, NaXiDoSp:11} based on computing eigenfunctions of local Dirichlet-to-Neumann maps. 
Other such spaces, which also can be used as approximation spaces independent of DD, are the so-called 
``multiscale" methods, which use ideas originally introduced to create basis functions for positive-definite problems with strongly varying coefficients 
\cite{HoWu:97,EfHo:09}; such multiscale methods for the Helmholtz equation include 
\cite{GaPe:15, Pe:17, BrGaPe;17, OhVe:18, ChVa:20, FuLiCrGu:21, MaAlSc:23, ChHoWa:23}. 
One attraction of using appropriate multiscale/problem-adapted basis functions for the coarse space is that, with certain parameter choices, some of these spaces are proved not to suffer from the pollution effect (at least, up to factors of $\log (kL)$); see, e.g., \cite{Pe:17, ChHoWa:23}. 

\subsection{\red{The current state of the art of frequency-explicit rigorous theory for two-level DD for high-frequency Helmholtz}}\label{sec:state}

%
The papers 
\cite{HuLi:24, LuXuZhZo:24, MaAlSc:24, FuGoLiWa:24} all analysed two-level ``hybrid" Schwarz preconditioners (i.e., the one-level solves and the coarse solves are combined in a multiplicative way) with multiscale coarse spaces. In \cite{MaAlSc:24, FuGoLiWa:24}, the preconditioned matrix is shown to be close to the identity, while in 
\cite{HuLi:24, LuXuZhZo:24} GMRES is shown to converge in a $k$-independent number of iterations (via appropriate bounds on the field of values of the preconditioned matrix).
These four two-level analyses 
all consider the Helmholtz interior impedance boundary value problem posed in a polygon/polyhedron. 
There are two issues with this:
\bit
\item[(a)] 
\red{Although the impedance boundary condition is used to approximate the radiation condition in the Helmholtz scattering problem, the error incurred by this approximation
is bounded away from zero, independently of $kL$, as $k\to \infty$ \cite{GLS1}.}
\item[(b)] 
When the interior impedance problem is  posed on a polygonal/polyhedral domain, 
it is difficult to prove that the Galerkin solution exists with fine spaces that have 
$\ll (kL)^{3d/2}$ degrees of freedom -- this is the case even if the domain is convex and the polynomial degree $p$  is large. This issue is discussed further in \S\ref{sec:other_people}; the difficulty is that  \red{standard arguments  showing well-posedness  of the Galerkin solution do not give the sharp conditions on $h$ and $k$ in terms of $p$ for polygonal/polyhedral boundaries when $p>1$.}
\eit 

The paper \cite{GS4} studies certain hybrid Schwarz preconditioners and gives sufficient conditions for 
the preconditioned matrix to be close to the identity, with this theory 
allowing DD subdomains of arbitrary size, and arbitrary absorbing layers/boundary conditions on both the global and local Helmholtz problems. The assumptions on the coarse space in \cite{GS4} are satisfied  (i) by the multiscale coarse spaces in the analyses \cite{HuLi:24, LuXuZhZo:24, MaAlSc:24, FuGoLiWa:24} and (ii) if the Galerkin problem in the coarse space is known to be quasi-optimal via the Schatz argument \cite{Sc:74, ScWa:96}. 
The theory in \cite{GS4} therefore covers piecewise-polynomial fine and coarse spaces, but only with fixed polynomial degree, \red{and therefore not pollution free.}

\red{The paper \cite{NaPa:24} proposes and analyses new GenEO-type coarse spaces, with the analysis crucially valid for non-Hermitian and/or indefinite problems.
While these coarse spaces can be therefore be applied to the high-frequency Helmholtz equation, there does not yet exist any associated $k$-explicit analysis.}

\subsection{\red{The main ideas and contribution of the present paper}}\label{sec:ideas}

Given that the $hp$-FEM does not suffer from the pollution effect, 
and, in contrast to multiscale methods, does not require any pre-computation of problem-adapted basis functions, \emph{the main motivation for the present paper is to propose and analyse the use of $hp$-FEM coarse spaces in \red{(additive)} two-level DD methods} for Helmholtz-type problems.  

\red{To achieve this, we first obtain, in Theorem \ref{thm:main_add}, a $k$- and $p$-explicit analogue of the classic two-level additive Schwarz analysis of Cai and Widlund \cite{CaWi:92} applied to the general Helmholtz-type PDE \eqref{eq:pdeIntro}. The main new technical ingredient is a $k$- and $p$-explicit stable splitting result (see \S\ref{sec:stable_splittings} below).} 
\red{Theorem \ref{thm:main_add}  is formulated under, essentially, the condition  that the Galerkin solution in the coarse space is quasi-optimal. We then apply Theorem \ref{thm:main_add} to the concrete setting of $hp$-FEM fine and coarse spaces 
applied to a Helmholtz problem (Definition \ref{def:CAP} below) that approximates a Helmholtz scattering problem (Definition \ref{def:scattering} below) 
with super-algebraic accuracy as $k\to \infty$. This result, Theorem \ref{thm:final} -- stated in an informal manner as
  Theorem \ref{thm:informal_pwp_add} below -- has the following three theoretically-desirable properties:}
%
    
%
\red{
\ben
\item (up to factors of $\log(kL)$) the fine and coarse spaces are both pollution free  (with the ratio of the coarse-space dimension to the fine-space dimension arbitrarily small),  
\item (up to factors of $\log(kL)$) the number of degrees of freedom per subdomain 
is constant, and
\item the number of GMRES iterations is bounded independently of $kL$. 
\een
As discussed in \S\ref{sec:state} (and also \S\ref{sec:other_people} below), none of the existing two-level convergence analyses recapped in \S\ref{sec:state} have Property 1. Furthermore, the existing two-level convergence analysis with pollution-free coarse spaces all involve problem-adapted basis functions. 

There are several sets of numerical experiments in the literature on two-level DD for Helmholtz and/or Maxwell with piecewise polynomial coarse spaces; see \cite{GSV1, GSV2, BoDoGrSpTo:17c, DDHelmholtz, DDMaxwell, BoDoJoTo:21}. Broadly speaking, these experiments show that, provided that the coarse space sufficiently resolves the oscillatory behaviour of the solution, the number of GMRES iterations of the two-level method grows slowly with $kL$. 
This paper therefore provides theoretical underpinning to these empirical observations; in particular, it gives  sufficient conditions for $kL$-independent GMRES iterations in the limit $kL\to\infty$.
In \S\ref{sec:discussion} we discuss in more detail the results of Informal Theorem \ref{thm:informal_pwp_add}/Theorem \ref{thm:final} in the context of the numerical experiments in \cite{BoDoJoTo:21}.
}

\subsection{\red{The Helmholtz problems covered by our theory}}

\red{
Theorem \ref{thm:main_add}  provides a  $k$- and $p$-explicit analogue of 
the results of Cai and Widlund \cite{CaWi:92}  
for the following general Helmholtz-type problem:~find $u\in H^1_0(\Omega)$ satisfying the PDE 
\beq\label{eq:pdeIntro}
-k^{-2} \nabla \cdot (A\nabla u) + k^{-1} B\cdot \nabla u + E u =f \quad\tin \Omega,
\eeq
where
$\Omega$ is  a bounded Lipschitz domain, $f\in (H^1_0(\Omega))^*$, and the coefficients $A,B,E$ are such that 
 $A \in L^\infty(\Omega,\Com^{d\times d})$, $B\in W^{1,\infty}(\Omega,\Com^d)$, and $E\in L^\infty(\Omega,\Com)$ with norms bounded above uniformly in $k$, and 
$A$ is Hermitian and positive definite, uniformly for $x\in \overline{\Omega}$ and for $k>0$.
 
 We highlight immediately that 
(i) this class of Helmholtz problems is the same as treated by Cai and Widlund  \cite{CaWi:92},  
(ii) Remark \ref{rem:why} below explains why the assumptions that $A$ 
is Hermitian and positive definite are needed for the theory (both in \cite{CaWi:92} and here), and 
(iii) using some of the arguments in  \cite{GSV1} (specifically  \cite[Lemmas 4.15 and 4.16]{GSV1}), this theory could 
be extended to problems with an impedance boundary condition on $\partial\Omega$; we do not do this because 
 of the drawbacks to using an impedance boundary condition described in Points (a) and (b) in \S\ref{sec:state}.
 }
  
%
\red{
As noted in \S\ref{sec:ideas}, Theorem \ref{thm:main_add} requires, essentially, that the Galerkin solution in the coarse space be quasi-optimal (just as in \cite{CaWi:92}). Giving sufficient conditions for this to hold with a piecewise-polynomial coarse space requires additional assumptions on the domain $\Omega$ and coefficients $A,B,$ and $E$. 
}
\red{The proof of the main result (Theorem \ref{thm:final}) verifies these assumptions with $hp$-FEM fine and coarse spaces and 
the $\Omega,A,B$, and $E$ coming from a particular approximation of the following scattering problem.}

\begin{definition}[Helmholtz scattering problem]\label{def:scattering}
Let  $A_{\rm scat} \in C^{\infty} (\Rea^d , \Rea^{d\times d})$, $d=2,3,$ be symmetric, positive-definite, and bounded in $\Rea^d$. Let $c_{\rm scat} \in C^\infty(\Rea^d;\Rea)$ be positive and bounded in $\Rea^d$. Furthermore, let $A_{\rm scat}$ and $c_{\rm scat}$ be such that 
${\rm supp}(I- A_{\rm scat})$ and ${\rm supp}(1-c_{\rm scat})$ are both compactly supported.
Given $f \in L^2(\Rea^d)$ with compact support 
and $k>0$, $v \in H^1_{\rm loc}(\Rea^d)$ satisfies the Helmholtz scattering problem if 
\begin{align*}
k^{-2}\nabla \cdot (A_{\rm scat}\nabla v) + c_{\rm scat}^{-2}v= -f\,\quad\text{ in } \Rea^d
\end{align*}
and 
\beqs
k^{-1}\partial_r v
(\bx) - \ri  v(\bx) = o \big(r^{-(d-1)/2}\big)
\quad \tas r:= |\bx|\tendi, \text{ uniformly in } \widehat x:= \bx/r.
\eeqs
\end{definition}

We approximate the solution of the Helmholtz scattering problem of Definition \ref{def:scattering} by the solution of the following problem \red{on a truncated domain, which falls into the class of Helmholtz problems 
\eqref{eq:pdeIntro}
discussed above}.

\begin{definition}\textbf{\emph{(Complex-absorbing-potential (CAP) approximation to Helmholtz scattering problem)}}\label{def:CAP}
Let $A_{\rm scat}$ and $c_{\rm scat}$ be as in Definition \ref{def:scattering}. 
Let $\Omega_{\rm int}$ be a
bounded Lipschitz open set containing 
${\rm supp}(I- A_{\rm scat})$ and ${\rm supp}(1-c_{\rm scat})$, and let $\Omega$ be a larger bounded Lipschitz polyhedron that strictly contains $\Omega_{\rm int}$. 
Let $V\in C^\infty(\Rea^d,\Rea)$ be non-negative, supported in $\Omega\setminus \Omega_{\rm int}$ and strictly positive in a neighbourhood of $\partial \Omega$. 
Given $f\in (H_0^{1}(\Omega))^*$ and $k>0$, $u\in H^1_0(\Omega)$ satisfies the CAP problem if 
\begin{align}\label{eq:CAP}
k^{-2}\nabla \cdot (A_{\rm scat} \nabla u) + \big(c^{-2}_{\rm scat} + \ri V\big)u= -f\,\quad\text{ in } \Omega.
\end{align}
\end{definition}

The weak form of \eqref{eq:CAP} is
\beq\label{eq:sesqui_CAP}
\int_\Omega k^{-2}(A_{\rm scat}\nabla u)\cdot\overline{\nabla w} - \big(c_{\rm scat}^{-2}+ \ri V\big) u\,\overline{w} = \langle f, w\rangle \quad\tfa w 
 \in H^1_0(\Omega).
\eeq
The CAP problem has a unique solution  (see Theorem \ref{thm:existence} below). Furthermore, 
the difference on $\Omega_{\rm int}$ between the solution of the scattering problem of Definition \ref{def:scattering} and the \red{solution of the} CAP problem is smooth and 
super-algebraically small in $k$ as $k\to\infty$  (see Theorem \ref{thm:CAP} below) -- this is the major advantage of CAP truncation compared to truncation by an impedance condition.
\red{Another way of approximating the radiation condition using complex absorption is perfectly matched layer (PML) truncation; the relative advantages and disadvantages between CAP and PML truncation are discussed in Remark \ref{rem:CAPvsPML} below.}

\subsection{The additive Schwarz preconditioner}\label{sec:1.4}

Let $\fine \subset H^1_0(\Omega)$ consist of Lagrange finite elements of degree $p_f$ (as defined e.g., in \cite[Chapter 3]{BrSc:08} or \cite[\S7.1]{ErGu:21}) on a shape-regular  
simplicial mesh  of diameter $h$; \red{i.e., there exists a constant $\Cshape$ such that $\rho_\tau \geq \Cshape h_\tau$ for all $\tau \in \fine$, where $\rho_\tau$ is the
  diameter of the largest inscribed sphere of any  element $\tau$ and $h_\tau$ is its diameter.}
  We call $\cV_h$ the \emph{fine space}.
Applying the Galerkin method to   \eqref{eq:pdeIntro} in the space $\cV_h$  yields  
  a linear system, with system matrix denoted  here  by $\matrixA$. 
  
  We construct domain-decomposition preconditioners for $\matrixA$ using a coarse space, $\coarse\subset\fine$, and a set of 
overlapping  subdomains $\{\Omega_\ell\}_{\ell=1}^N$. 
Let  $\subdomain:= \fine \cap H^1_0(\Omega_\ell)$ (with freedoms in the interior of each $\Omega_\ell$); i.e., we impose zero Dirichlet boundary conditions on the subdomain problems.

Let $\coarse\subset H^1_0(\Omega)$ consist of Lagrange finite elements \red{of degree $p_c$} on a mesh of width $\Hcoarse$ \red{with the coarse mesh elements resolved by the fine mesh and $p_c\leq p_f$.}

  Let $\{\phi_j\}_{j\in \cJ_h}$ denote \red{a Lagrange basis for $\cV_h$ (with nodes $\{x_i\}_{i\in \cJ_h}$, where $\cJ_h$ is a suitable index set),} 
and $\{\Phi_q\}_{q\in \cJ_0}$ denote a Lagrange basis for $\coarse$ (with suitable index set $\cJ_0$). Then, 
 for all $q \in \cJ_0$,
\beq\label{eq:Phi1}
\Phi_p = \sum_{j \in \cJ_h} (\matrixR_0)_{qj} \phi_j, \quad \text{where} \quad (\matrixR_0)_{qj} = \Phi_q(x_j). 
\eeq
The matrix $\matrixR_0^T$ then maps the freedoms of any function in $\coarse$ to its freedoms in $\fine$. Similarly, let $\matrixR_\ell^T$ be the usual extension matrix that maps the freedoms of any $v_{h,\ell}\in \cV_\ell$ to its freedoms in $\fine$ (via padding by zeros) and $\matrixR_\ell= (\matrixR_\ell^T)^T$. 
The matrices 
$\matrixA_0:= \matrixR_0  \matrixA \matrixR_0 ^T$, and $\matrixA_\ell:= \matrixR_\ell  \matrixA \matrixR_\ell ^T$ are then Galerkin matrices of $a(\cdot,\cdot)$ discretised in $\cV_0$ and $\cV_\ell$, respectively. 

We \red{analyse}
the purely additive preconditioner for $\matrixA$,
\beq\label{eq:matrixB_add}
\matrixBadd^{-1}= \matrixBadd^{-1}(\matrixA):= 
\matrixR_0^T\matrixAzero^{-1} \matrixR_0 
+
\sum_{\ell=1}^N
\matrixR_\ell^T\matrixA_\ell^{-1} \matrixR_\ell,
\eeq
as both a left and a right preconditioner. \red{Our results are proved in the Euclidean inner product weighted by} the real symmetric positive-definite matrix $\matrixD_k\in \Rea^{n\times n}$ \red{defined by} \begin{equation}\label{eq:innerproducts}
\big(v_h, w_h \big)_{H^1_k(\Omega)} =\big\langle \bV,  \bW\big\rangle_{\matrixD_k},
\eeq
for all $v_h, w_h \in\fine$ with freedoms $\bV, \bW$.

\subsection{Informal statement of the main result \red{for the CAP problem}}\label{sec:intro_poly}


\begin{informaltheorem}\mythmname{\red{GMRES convergence for solving the finite element approximation of the CAP problem preconditioned by
      \eqref{eq:matrixB_add}}}
\label{thm:informal_pwp_add}Suppose that $\Omega$ is convex \red{(since $\Omega$ is a user-chosen truncation domain, this assumption is not restrictive)},
\bit
\item the fine space consists of degree-$p_f$ Lagrange finite elements on a quasi-uniform mesh of diameter $h$ and 
\item the subdomains $\{\Omega_\ell\}_{\ell=1}^N$ have generous overlap; i.e., $\delta\sim \Hsub$ (where $\delta$ is the parameter related to the minimum overlap of the subdomains, and $\Hsub$ is the maximum subdomain diameter), and
\item the boundaries of the subdomains are resolved by the fine mesh.
\eit 

\red{
Suppose further that the solution operator of the CAP problem is bounded polynomially in $kL$ as $kL\to\infty$. 
Then, given $\Cdegreet\geq \Cdegreec >0$ there exist $c_1, c_2>0$ such that for all $\Ccoarse>1$,
\begin{align} \label{Ivan1}
h =\frac{c_1}{\Ccoarse k}, \quad p_f \sim 1 + \Cdegreet \log (kL),
\end{align}
\begin{align}\label{Ivan2}
\Hcoarse= \frac{c_1\big(1 + \Cdegreet \log (kL)\big)}{k},
\quad p_c \sim 1 + \Cdegreec \log (kL)  \text{ with }  p_c\leq p_f,\,\tand\,\Hsub = \frac{c_2}{k} 
\end{align}
(i.e., $h\sim k^{-1}, \Hcoarse\sim k^{-1}\log (kL)$, $\Hsub\sim k^{-1}$) 
then
when GMRES \red{(without restarts)} is applied to \emph{either} 
$\matrixBadd^{-1}\matrixA$ in the $\matrixD_k$ inner product 
\emph{or} $\matrixA\matrixBadd^{-1}$ in the $\matrixD_k^{-1}$ inner product the number of iterations is bounded independently of $kL$ as $kL\to \infty$.
}
\end{informaltheorem}

\red{Note that \eqref{Ivan1} and \eqref{Ivan2} implies $\Hcoarse>\Ccoarse h$ with $\Ccoarse>1$. Furthermore,
\beqs
\frac{\text{coarse-space dimension }
}{\text{fine-space dimension }
} 
\sim \bigg(\frac{h}{\Hcoarse}\frac{p_c}{p_f}\bigg)^d 
\sim \bigg(\frac{\Cdegreec}{\Cdegreet}\frac{1}{\Ccoarse \big(1 + \Cdegreet \log(kL)\big)}\bigg)^{d}\to 0\,\tas kL\to \infty,
\eeqs
and so  the dimension reduction can be improved by  increasing either $\Ccoarse$ or   $ \Cdegreet/\Cdegreec$.}

The precise statement of Informal Theorem \ref{thm:informal_pwp_add} is Theorem \ref{thm:final} below. 

\subsection{Discussion of Informal Theorem \ref{thm:informal_pwp_add}}\label{sec:discussion}

\paragraph{The \red{number of degrees of freedom per subdomain}.}
\red{Recall that keeping the number of degrees of freedom in each subdomain constant and then increasing the number of subdomains
is a popular strategy to seek parallel scalability as the total number of degrees of freedom of the problem increases. Informal Theorem \ref{thm:informal_pwp_add} is in this regime, up to factors of $\log k$. To see this, observe that \eqref{Ivan1} and \eqref{Ivan2} implies that $\Hsub/h$ is constant, and thus 
the number of degrees of freedom on each subdomain $\sim (p_f/h)^d \Hsub^{d} \sim p_f^d \sim (\log k)^d$; i.e., grows logarithmically with $k$. 
}

\paragraph{The coarse space needs to resolve the propagative behaviour of the solution.}
One expects that a one-level DD method with subdomains of size $\sim k^{-1}$ needs at least $\sim k$ iterations to see the propagation of the Helmholtz solution operator at length scales independent of $k$ (and this is borne out in numerical experiments; see, e.g., \cite[Table 4]{GSV1}).
To obtain a $k$-independent number of iterations, the coarse space must therefore resolve this propagation.

For example, the numerical experiments of \cite{BoDoJoTo:21} (discussed in more detail below) considered coarse spaces with $\Hcoarse=2h$ 
When the fine grid had 10 grid points per wavelength, the coarse solve dramatically decreased the number of GMRES iterations (compared to the one-level method) \cite[Table 7]{BoDoJoTo:21}, but this strategy failed when the   fine grid had  5 grid points per wavelength  \cite[Table 5]{BoDoJoTo:21}.

This requirement that the coarse space must resolve the propagation is encoded in Theorem \ref{thm:main_add} as Assumption \ref{ass:coarse}. This is the same requirement on the coarse space as in the theory in \cite{GS4}, and 
 the multiscale coarse spaces of 
\cite{HuLi:24, LuXuZhZo:24, MaAlSc:24, FuGoLiWa:24} all satisfy this assumption.
%
Thus, like Theorem \ref{thm:informal_pwp_add}, 
\cite{HuLi:24, LuXuZhZo:24, MaAlSc:24, FuGoLiWa:24,GS4}, all require the Galerkin solution in the coarse space to be   quasi-optimal.

\paragraph{Near-pollution-free fine and coarse spaces when $p_f$ and $p_c \sim \log k$.}
In Theorem \ref{thm:informal_pwp_add}, the fine space has dimension $\sim (kL)^d (\log (kL))^{d}$ and the coarse space has dimension $\sim (kL)^d$  (i.e., both spaces are pollution free up to logarithmic factors).

\red{The $hp$-FEM convergence results in \cite{MeSa:10, MeSa:11, EsMe:12, LSW3, LSW4, GLSW1, BeChMe:24, GSAN} all involve (fine) spaces of dimension $\sim (kL)^d$. The reason for the extra factor of $(\log (kL))^{d}$ in our theory is explained in Remark \ref{rem:QI} below.}

\red{The important question of how to efficiently solve a coarse problem with $\sim (kL)^d$ degrees of freedom}
 is not explored in the present paper \red{(nor in the rigorous analyses in \cite{HuLi:24, LuXuZhZo:24, MaAlSc:24, FuGoLiWa:24,GS4})}, but we note that
special techniques for solving such a coarse problem with a one-level method are investigated numerically in   \cite{ToJoDoHoOpRi:22}, \cite[\S6]{BDGST} with some success.

\paragraph{Relation to the numerical experiments in \cite{BoDoJoTo:21}.}\, \red{Theorem \ref{thm:informal_pwp_add} gives theoretical insight into existing numerical results on two-level DD preconditioners using piecewise polynomials. Indeed,} 
the experiments in \cite{BoDoJoTo:21}, which consider $p_c=p_f=2$ and a hybrid Schwarz preconditioner, show that the number of GMRES iterations 
(i) grows slowly with $k$ 
when the number of degrees of freedom per subdomain is kept constant, and 
(ii) grows with $k$ if the coarse space does not resolve the oscillatory/propagative nature of the solution.
\red{Specific details about these two points is given in Appendix \ref{sec:PHT} below.}

\subsection{Plan of the paper}
\S\ref{sec:assumptions} states and discusses the assumptions needed to prove Theorem \ref{thm:main_add}, which is then stated with corollaries in \S\ref{sec:main}. 
\S\ref{sec:aux} gives auxilliary results needed for the proof of Theorem \ref{thm:main_add}.  
\S\ref{sec:proof_abstract_add} proves Theorem \ref{thm:main_add}.
\S\ref{sec:CAP} gives results about the Helmholtz CAP problem. 
\S\ref{sec:pwp} applies Theorem \ref{thm:main_add} to \red{the CAP problem with $hp$-FEM fine and coarse spaces} and gives the 
precise statement of Informal Theorem \ref{thm:informal_pwp_add} (Theorem \ref{thm:final}) \red{and its proof}.
\S\ref{sec:appendix} \red{gives the operator representation of the preconditioner \eqref{eq:matrixB_add}.}
\S\ref{app:CAP} proves Theorems \ref{thm:nt_CAP} and \ref{thm:CAP} (auxiliary results about the CAP problem).

\section{Statement of the assumptions}\label{sec:assumptions}

\subsection{Assumptions on the finite-element space and domain decomposition}\label{sec:ass_fine}

\begin{assumption}[The fine space]\label{ass:fine}
$\Omega$ is a Lipschitz polyhedron and
 $\cT^h$ is  a family of
conforming simplicial meshes on $\Omega$ 
that are shape regular \red{with constant $\Cshape$} as the  
mesh diameter $h \rightarrow 0$.  
The fine space $\fine \subset H^1_0(\Omega)$ consists of piecewise polynomials on $\cT^h$ of degree $p_f$. 
\end{assumption}

\red{
We now use notation taken from \cite[\S2.2]{Vo:25} to define a patch of elements (see \cite[Figure 2]{Vo:25} for figure illustrating this notation).

\begin{definition}[Two-neighbour element patch]\label{def:patch}
Given $K\in \mathcal T_h$, let $\widetilde{\mathcal{T}}_K$ be the union of all elements neighbouring $K$, and all neighbours of these neighbours (i.e., two layers of elements around $K$).
\end{definition}


}


\begin{definition}[Characteristic length scale]\label{def:char}
A domain has characteristic length scale $L$ if its diameter $\sim  L$, its surface area $\sim  L^{d-1}$, and its volume $\sim L^d$ (where the hidden constants in $\sim$ are independent of $L$).
\end{definition} 

\begin{assumption}[The subdomains]\label{ass:subdomain_def}
The subdomains $\{ \Omega_\ell \}^N_{\ell =1}$ form an overlapping cover of $\Omega$, with each $\Omega_\ell$ a non-empty open polyhedron with characteristic length scale $H_\ell$ \red{such that $\overline{\Omega_\ell}$} is a union of elements of $\cT^h$. Let $h_\ell  := \max_{\tau  \subset \overline{\Omega_\ell}}  h_\tau$, where $h_\tau$ is the diameter of $\tau\in \cT^h$.

$\{\chi_\ell\}_{\ell=1}^N$ is a partition of unity subordinate to $\{ \Omega_\ell \}^N_{\ell =1}$ \red{(i.e., $0\leq \chi_\ell(x)\leq 1$ for all $x$, $\supp \chi_\ell \subset \Omega_\ell$, and $\sum_{\ell=1}^N\chi_\ell(x)=1$ for all $x\in\Omega$)}. 
Each $\chi_\ell$  is continuous on $\Omega$ and piecewise linear on $\cT^h$. 
Furthermore,
there exists $\Cpou>0$ such that for all $\ell=1,\ldots, N$, 
there exists $\delta_\ell>0$ such that
\beq\label{eq:PoU}
\N{\nabla \chi_\ell}_{L^\infty(\tau)}\leq \Cpou \delta_\ell^{-1} \quad \tfa \tau \in \cT^h.
\eeq
\red{
  In addition, if $K\in (\Omega_\ell)^c$, then $\widetilde{\mathcal{T}}_K\cap \supp\chi_\ell=\emptyset$ (informally, the distance of $\supp \chi_\ell$  from $\partial \Omega_\ell\setminus \partial\Omega$ is greater than two elements).


}
\end{assumption}

The quantities $\delta_\ell, \ell=1,\ldots,N$, are indicators of the size of the overlap of the subdomains $\Omega_\ell$ (e.g., if the overlaps $\to 0$ then the $\delta_\ell \to 0 $). We introduce $\delta_\ell$ via \eqref{eq:PoU}, since this is the property that is actually used in the proofs (see Lemmas \ref{lem:split} and \ref{lem:split_two} below).

In 
 \cite[Lemma 3.4]{ToWi:05} there is an explicit construction of 
a partition of unity satisfying the conditions in Assumption \ref{ass:subdomain_def} apart from the condition that 
the distance of $\supp \chi_\ell$  from $\partial \Omega_\ell\setminus \partial\Omega$ is greater than two elements; 
the construction in  \cite[Lemma 3.4]{ToWi:05} can be easily modified to satisfy this additional condition.

Let 
\beq\label{eq:Hsub}
\Hsub:= \max_{\ell}H_\ell, \quad \delta:= \min_{\ell} \delta_\ell,
\eeq
 and 
let 
\beq\label{eq:Lambda}
\Lambda  := \max\bigl\{ \#\Lambda (\ell ):\ell =1,...,N\bigr\},\quad\text{where}\quad \Lambda (\ell )=\big\{ \ell^\prime  :\Omega_\ell \cap \Omega_{\ell^\prime}  \not =\emptyset \big\};
\eeq
i.e., $\Lambda$ is the maximum number of subdomains that can overlap any given subdomain.

%
\begin{assumption}[Piecewise-polynomial coarse space]\label{ass:coarse_new}
 $\cT^\Hcoarse$ is  a family of
conforming simplicial meshes on $\Omega$ (with affine element maps) 
that are shape regular \red{with constant $\Cshape$} 
as the  
mesh diameter $\Hcoarse \rightarrow 0$ \red{and additionally such that each coarse element is comprised of a union of elements of $\cT^h$.}   
The coarse space $\coarse \subset \red{\fine\subset} H^1_0(\Omega)$ consists of piecewise polynomials on $\cT^\Hcoarse$ of degree $1\leq p_c\leq p_f$.
\end{assumption}

\subsection{Assumptions on the sesquilinear form}\label{sec:ass}

\red{
Our general theory applies to the sesquilinear form $a(\cdot,\cdot)$ corresponding to the problem \eqref{eq:pdeIntro}.
}

\begin{assumption}[The sesquilinear form]\label{ass:sesqui}
Let $A \in L^\infty(\Omega,\Com^{d\times d})$, $B\in W^{1,\infty}(\Omega,\Com^d)$, and $E\in L^\infty(\Omega,\Com)$ with norms bounded above uniformly in $k$. Furthermore, 
$A$ is Hermitian and positive definite, uniformly for $x\in \overline{\Omega}$ and for $k>0$.
For all $u,v\in H^1_0(\Omega)$,
\beq \label{Ivan3}
a(u,v) := \int_\Omega \Big( k^{-2} (A \nabla u)\cdot \overline{\nabla v} + k^{-1}( B \cdot \nabla u) \overline{v}+ E u \overline{v}\Big).
\eeq
\end{assumption}

\begin{example}\label{ex:CAP}
The  sesquilinear form of the CAP problem  \eqref{eq:sesqui_CAP} is of the form  \eqref{Ivan3} with $A=A_{\rm scat}$, $B=0$, and $E= -c_{\rm scat}^{-2} - \ri V-1$.
\end{example}

The conditions on $A$ in Assumption \ref{ass:sesqui} imply that 
\beq\label{eq:1kip}
(u,v)_{H^1_k(D)} := k^{-2}\big( A \nabla u , \nabla v\big)_{L^2(D)} + (u,v)_{L^2(D)}
\eeq
is an inner product 
on $H^1(D)$ with the associated norm
\beq\label{eq:1knorm}
\|u\|^2_{H^1_k(D)} := k^{-2}\|A^{1/2} \nabla u\|^2_{L^2(D)} + \|u\|^2_{L^2(D)}.
\eeq

Assumption \ref{ass:sesqui} implies that 
there exist $C_{\rm cont}, c_{\rm G}, C_{\rm G}>0$ such that, 
for all $u,v\in  H^1_0(\Omega)$ and $k>0$,
\begin{align}
\big|a(u,v)\big| \leq C_{\rm cont} \N{u}_{H^1_k(\Omega)}\N{v}_{H^1_k(\Omega)}\quad \tand \quad
\label{eq:Garding}
\Re a(v,v) \geq \Cgo c_{\rm G}\N{v}^2_{H^1_k(\Omega)} - C_{\rm G} \N{v}^2_{L^2(\Omega)}. 
\end{align}

\red{
Applying the Galerkin method to  variational problems involving the sesquilinear form $a(\cdot,\cdot)$ \eqref{Ivan3} in the space $\cV_h$  yields  
  a linear system, whose system matrix we denote  by $\matrixA$. 
}


\subsection{The coarse and subdomain operators $Q_\ell$}

\subsubsection{Statement of the assumptions.}

For $\ell=1,\ldots, N$, let $\cV_\ell := \fine \cap H^1_0(\Omega_\ell)$.
For $\ell=0,\ldots,N$,  we define $Q_\ell : \fine 
\to \subdomain$  by 
\beq\label{eq:Q}
a\big( Q_\ell v_h, w_{h,\ell} \big)= a\big( v_h, w_{h,\ell} \big)\,\,\tfa w_{h,\ell} \in \subdomain.
\eeq
\red{The following two assumptions concern the existence and other properties of $Q_\ell, \ell=0,\ldots,N$.}

\begin{assumption}[Bounds on coarse-space Galerkin error]
\label{ass:coarse}
There exists a unique operator 
$Q_0:\fine\to \coarse$ satisfying \eqref{eq:Q} and 
there exist
$\Cnormo,\sigma>0$ such that, for all $v_h\in \fine$, 
\beq\label{eq:QO2a}
\N{(I-Q_0)v_h}_{H^1_k(\Omega)} \leq \Cnormo\N{v_h}_{H^1_k(\Omega)}
\quad\tand\quad
\N{(I-Q_0)v_h}_{L^2(\Omega)} \leq  \sigma\N{v_h}_{H^1_k(\Omega)}.
\eeq
\end{assumption}

We emphasise that \red{Theorem \ref{thm:main_add} below requires that} $\sigma$ in \eqref{eq:QO2a} should be sufficiently small -- see \eqref{eq:newHsub} below.
For the $hp$-FEM coarse spaces that we consider here, this is ensured by proving that the coarse-space solution is quasi-optimal via the Schatz argument \cite{Sc:74} (see, e.g., \cite{MeSa:10}, \cite[Appendix B]{GS4}). 

\begin{assumption}[Boundedness of the subdomain operators $Q_\ell$]\label{ass:subdomain}
There exists a unique operator 
$Q_\ell:\fine\to \subdomain$ satisfying \eqref{eq:Q} and, given $k_0>0$, there exists $\Cnormell>0$ such that, for all $k\geq k_0$, $\ell=1,\ldots,N$, and 
$v_h\in \fine$, 
\beq\label{eq:Qellbound}
\N{Q_\ell v_h}_{H^1_k(\Omega_\ell)}\leq \Cnormell\N{v_h}_{H^1_k(\Omega_\ell)}.
\eeq
\end{assumption}

%
%

\subsubsection{Discussion of Assumption \ref{ass:subdomain}}

  One way to satisfy Assumption \ref{ass:subdomain} is for $a(\cdot,\cdot)$ to be 
coercive on $H^1_0(\Omega_\ell)$. 

\ble\label{lem:subdomain_coercive}
Suppose that $a(\cdot,\cdot)$ satisfies Assumption \ref{ass:sesqui} and, in addition, 
$a(\cdot,\cdot)$ is 
coercive when restricted to $H^1_0(\Omega_\ell)$, with 
coercivity constant independent of $k$; 
i.e., there exists $c>0$ such that
\beq\label{eq:coercivity}
\big| a(v,v)\big| \geq c \N{v}_{H^1_k(\Omega)}^2 \quad\tfa v \in H^1_0(\Omega_\ell).
\eeq
Then Assumption \ref{ass:subdomain} holds.
\ele

\bpf
Assumption \ref{ass:sesqui} implies that $a(\cdot,\cdot)$ is continuous on $H^1_0(\Omega_\ell)$ and, furthermore,
for all $u,v \in H^1_0(\Omega)$ with at least one of them in $H^1_0(\Omega_\ell)$ 
\beq\label{eq:cont2}
\big|a(u,v)\big| \leq \Ccont \N{u}_{H^1_k(\Omega_\ell)}\N{v}_{H^1_k(\Omega_\ell)}.
\eeq
By \eqref{eq:cont2}, 
for all $v_h\in \fine$, the map $w_{h,\ell}\mapsto a(v_h,w_{h,\ell})$ is an anti-linear functional on the space $\subdomain$ with the norm $\|\cdot\|_{H^1_k(\Omega_\ell)}$. 
Continuity, coercivity, and the Lax--Milgram lemma applied with the Hilbert space $\subdomain$ with the norm $\|\cdot\|_{H^1_k(\Omega_\ell)}$ 
imply the solution to \eqref{eq:Q} for $\ell=1,\ldots,N$ exists, i.e., $Q_\ell : \fine \to \subdomain$ is well-defined. 
By (in this order) coercivity on $H^1_0(\Omega_\ell)$ \eqref{eq:coercivity}, the fact that $Q_\ell v\in \subdomain\subset H^1_0(\Omega_\ell)$, the equation \eqref{eq:Q} defining $Q_\ell$, 
the fact that $Q_\ell v\in H^1_0(\Omega_\ell)$, 
and the property \eqref{eq:cont2}, for all $v_h\in \fine$, 
\begin{align*}
\N{Q_\ell v_h}^2_{H^1_k(\Omega_\ell)}\leq C \big| a(Q_\ell v_h, Q_\ell v_h)\big| = C\big| a(v_h, Q_\ell v_h)\big| &\leq C\Ccont \N{v_h}_{H^1_k(\Omega_\ell)}\N{Q_\ell v_h}_{H^1_k(\Omega_\ell)},
\end{align*}
and the bound \eqref{eq:Qellbound} follows.
\epf

\red{
\bre
Assumption \ref{ass:subdomain} would also be satisfied if, rather than being coercive,  
$a(\cdot,\cdot)$ merely satisfied a discrete inf-sup condition on $\subdomain$. The constant $\Cnormell$ in \eqref{eq:Qellbound} would then be the inverse of the discrete inf-sup constant. However, in the course of the proof of Theorem \ref{thm:main_add}, $k\Hsub$ will be made small anyway (see Lemma \ref{lem:GSV} below), and when $k\Hsub$ is sufficiently small, $a(\cdot,\cdot)$ is coercive on $H^1_0(\Omega_\ell)$ by the Poincar\'e inequality (Theorem \ref{thm:Poincare} below). We therefore only seek to satisfy Assumption \ref{ass:subdomain} via coercivity. 
\ere
}


\section{\red{Our $k$- and $p$-explicit analogue of the main result of \cite{CaWi:92}}}\label{sec:main}

Given $Q_\ell, \ell=0,\ldots,N$, defined by \eqref{eq:Q}, let 
\beq\label{eq:pc}
\Qa:= Q_0 + \sum_{\ell=1}^N Q_\ell.
\eeq
\red{As in \S\ref{sec:1.4}, let $\{\phi_j\}_{j\in \cJ_h}$ denote a Lagrange basis for $\cV_h$. Let $\matrixA$ be the Galerkin matrix of the sesquilinear form $a(\cdot,\cdot)$ \eqref{Ivan3} discretised in $\fine$; i.e.,
$$
\matrixA_{j\ell} = a(\phi_\ell, \phi_j).
$$
Let  $\matrixB_a^{-1}$ be then defined by   \eqref{eq:matrixB_add}.}
Then $Q$ represents 
the preconditioned
  matrix 
  $\matrixB_a^{-1}\matrixA$, 
 in the sense that 
\begin{equation}\label{eq:matrixQ}
\bigl(\Qa v_h, w_h \bigr)_{H^1_k(\Omega)}=\bigl\langle \matrixBadd^{-1}\matrixA\bV,  \bW\bigr\rangle_{\matrixD_k}
\end{equation}
for all $v_h, w_h \in\fine$ with freedoms $\bV, \bW$; see \S\ref{sec:appendix}.

\begin{theorem}[Upper and lower bounds on the field of values of $\Qa$]\label{thm:main_add}
Suppose the assumptions in \S\ref{sec:assumptions} hold. 
Given the constants $\Cpou, \Cnormo$, and $\Cnormell$ in the assumptions in \S\ref{sec:assumptions} 
  and $k_0$, there exists $C_1, C_2, C_3>0$ such that the following holds.
Suppose that $k, \Lambda, p_f, \Hcoarse, h, \Hsub,$ and $\delta$ satisfy the following conditions:~$k\geq k_0$, $\Lambda, p_f\in\mathbb{Z}^+$, and 
\beq\label{eq:newHsub}
\red{\max\big\{\sigma, k\Hsub\big\} \Lambda^{1/2}
(1+kh)
\bigg[
\bigg(1 +\frac{k\Hcoarse}{p_c}\bigg)^2+ (k\delta)^{-2} \bigg(\frac{k\Hcoarse}{p_c}\bigg)^2\bigg]^{1/2} 
\leq C_1.
}
\eeq
Then, for all $v_h\in \fine$, 
\beq\label{eq:upper_add1}
\big\|\Qa v_h\big\|_{H^1_k(\Omega)}\leq C_2 \Lambda\N{v_h}_{H^1_k(\Omega)}
\eeq
and
\beq\label{eq:upper_add2}
\frac{
\big|
( v_h,\Qa v_h)_{H^1_k(\Omega)}
\big|
}{
\N{v_h}_{H^1_k(\Omega)}^2 
}
\geq
\red{ C_3 \Lambda^{-1}
(1+kh)
^{-2} \bigg[
\bigg(1 +\frac{k\Hcoarse}{p_c}\bigg)^2+ (k\delta)^{-2} \bigg(\frac{k\Hcoarse}{p_c}\bigg)^2\bigg]^{-1}.
}
\eeq
\end{theorem}

\red{
\bre
By the order of quantifiers in Theorem \ref{thm:main_add}, $C_1, C_2, C_3$ can depend on $\Cpou, \Cnormo$, $\Cnormell$, and $k_0$, but not on $k, \Lambda, p_f, \Hcoarse, h, \Hsub,$ and $\delta$. The dependencies in other statements in the paper can be inferred in a similar way.
\ere
}

The bounds \eqref{eq:upper_add1} and \eqref{eq:upper_add2}  
then imply results about GMRES applied to   $\matrixBadd^{-1}\matrixA$.

\begin{corollary}[Bound on the norm and field of values of $\matrixB_a^{-1}\matrixA$]
\label{cor:1}
Under the assumptions of Theorem \ref{thm:main_add}, for all $\bV \in\Com^{n}$,
\beqs
\big\|\matrixBadd^{-1}\matrixA\big\|_{\matrixD_k}\leq C_2 \Lambda
\,
\tand
\,
\frac{
\big|
\langle \bV,  \matrixBadd^{-1}\matrixA\bV\rangle_{\matrixD_k}
\big|
}{
\N{\bV}_{\matrixD_k}^2
}
\geq 
\red{C_3
(1+kh)
^{-2} \bigg[
\bigg(1 +\frac{k\Hcoarse}{p_c}\bigg)^2+ (k\delta)^{-2} \bigg(\frac{k\Hcoarse}{p_c}\bigg)^2\bigg]^{-1}.
}
\eeqs
\end{corollary}

By the Elman-type estimate \cite{BeGoTy:06} for weighted GMRES (see \cite[Corollary 5.4]{BDGST}), Corollary \ref{cor:1} implies the following. 

\begin{corollary}[Convergence of GMRES]\label{cor:2}
There exists $C>0$ such that the following is true. Given $\epsilon>0$, 
under the assumptions of Theorem \ref{thm:main_add}, if 
\beq\label{eq:it_bound}
m\geq C \frac{C_2}{C_3} \Lambda 
\red{(1+kh)
^{2} \bigg[
\bigg(1 +\frac{k\Hcoarse}{p_c}\bigg)^2 + (k\delta)^{-2} \bigg(\frac{k\Hcoarse}{p_c}\bigg)^2\bigg]
}
 \log\bigg(\frac{12}{\epsilon}\bigg),
\eeq 
then when GMRES \red{(without restarts)} is applied to $\matrixBadd^{-1}\matrixA$  in the $\matrixD_k$ inner product, the $m$th relative residual is $\leq \epsilon$.
\end{corollary}

\bre[Weighted vs unweighted GMRES]
\cite[Corollary 5.8]{GoGrSp:20} showed (via an inverse estimate) that 
if the fine mesh sequence $\cT^h$ is quasiuniform then GMRES applied in the Euclidean inner product with the same initial residual takes at most an extra $C\log (kh)^{-1}$ iterations to ensure the same relative residual as if GMRES were applied in the $\matrixD_k$ weighted inner product 
(see the last displayed equation in the proof of \cite[Corollary 5.8]{GoGrSp:20}).
The numerical experiments in 
\cite[Experiment 1]{GSV1}, \cite[\S6]{BDGST} showed little difference in the number of weighted/unweighted iterations.
\ere

\begin{corollary}[Results for right-preconditioning]\label{cor:3}
Suppose that the assumptions of  Theorem \ref{thm:main_add} hold, except that Assumptions \ref{ass:coarse} and \ref{ass:subdomain} hold with the sesquilinear form $a(\cdot,\cdot)$ 
\red{in the definition of $Q_\ell, \ell=0,\ldots,N$, \eqref{eq:Q}} replaced by its adjoint (i.e., $a^*(u,v):= \overline{a(v,u)}$). Then for all $\bW \in\Com^{n}$,
$\big\|\matrixA \matrixBadd^{-1}\big\|_{\matrixD_k^{-1}}\leq C_2 \Lambda$ and 
\beqs
\frac{
\big|
\langle  \matrixA \matrixBadd^{-1}\bW, \bW\rangle_{\matrixD_k^{-1}}
\big|
}{
\N{\bW}_{\matrixD_k^{-1}}^2
}
\geq C_3\Lambda^{-1} 
\red{(1+kh)
^{-2} \bigg[
\bigg(1 +\frac{k\Hcoarse}{p_c}\bigg)^2+ (k\delta)^{-2} \bigg(\frac{k\Hcoarse}{p_c}\bigg)^2\bigg]^{-1}.
}
\eeqs
Thus, given $\epsilon>0$, if $m$ satisfies \eqref{eq:it_bound} 
then when GMRES \red{(without restarts)} is applied to $\matrixA \matrixBadd^{-1}$  in the $\matrixD_k^{-1}$ inner product, the $m$th relative residual is $\leq \epsilon$. 
\end{corollary}
\red{Corollary \ref{cor:3} is obtained from Corollary \ref{cor:1} by arguing as in \cite[Proof of Theorem 5.8]{GSV1}.}

\section{Auxilliary results needed for the proof of Theorem \ref{thm:main_add}}\label{sec:aux}

\subsection{Consequences of the definition of $\Lambda$}

\begin{lemma}
For all $v\in L^2(\Omega)$ and $w\in H^1(\Omega)$, 
\beq\label{eq:overlap2a}
\sum_{\ell=1}^N \N{v}^2_{L^2(\Omega_\ell)} \leq \Lambda \N{v}^2_{L^2(\Omega)}
\quad\tand\quad 
\sum_{\ell=1}^N \N{w}^2_{H^1_k(\Omega_\ell)} \leq \Lambda \N{w}^2_{H^1_k(\Omega)}. 
\eeq
Furthermore, given $v_\ell\in \subdomain$,
\beq\label{eq:overlap_new}
\bigg\|
\sum_{\ell=1}^N v_\ell 
\bigg\|^2_{H^1_k(\Omega)}
\leq 2\Lambda \sum_{\ell=1}^N \N{v_\ell}^2_{H^1_k(\Omega_\ell)}.
\eeq
\end{lemma}

\bpf
The bounds \eqref{eq:overlap2a} follow immediately from the definition \eqref{eq:Lambda} of $\Lambda$.
The bound \eqref{eq:overlap_new} without an explicit expression for the constant is proved in \cite[Lemma 4.2]{GSV1}. The definition of $\Lambda$ implies that \cite[Equation 4.8]{GSV1} holds with $\lesssim \sum_{\ell=1}^N \|v_\ell\|^2_{H^1_k(\Omega)}$ at the end replaced by $\leq \Lambda \sum_{\ell=1}^N \|v_\ell\|^2_{H^1_k(\Omega)}$. The result then follows, with the factor of $2$ arising from use of the inequality $(a+b)^2 \leq 2 a^2 + 2b^2$ at the end of  \cite[Proof of Lemma 4.2]{GSV1} (with this constant hidden in the notation $\lesssim$ in \cite[Proof of Lemma 4.2]{GSV1}). 
\epf

\subsection{\red{Recap of degree-explicit quasi-interpolation results}}\label{sec:poly}

\red{
We use two different quasi-interpolants with degree-explicit approximation bounds. The first, from \cite{Vo:25}, leaves the finite-element space invariant (and we crucially need this property), whereas the second, from \cite{KaMe:15}, does not.
We use the first to approximate in $\fine$, with the approximation bounds used element-wise, whereas we use the second to approximate in $\coarse$, and only need approximation on all of $\Omega$. The results from \cite{Vo:25} and \cite{KaMe:15} are therefore stated in these setups.}

\red{
\begin{theorem}[A projection quasi-interpolant with degree-independent
   bounds from \cite{Vo:25}]\label{t:QIV} 
Given $\Cshape>0$ there exists $C>0$ 
such that for  $\fine\subset H^1_0(\Omega)$ satisfying Assumption \ref{ass:fine}  
there exists a linear operator  $\inter_h^{\rm V}: H^1_0(\Omega)\to \fine$ such that 

(i) $\inter_h^{\rm V} v_h= v_h$ for all $v_h\in \fine$,

(ii) for all $w\in H^1_0(\Omega)$ and $K\in \mathcal{T}^h$,
\beq\label{e:QIV1}
\| \nabla\big(( I- \inter_h^{\rm V}) w)\|_{L^2(K)} \leq C \sum_{K' \in \widetilde{\mathcal{T}}_K } \| w\|_{H^1(K')}, \quad\tand
\eeq

(iii) 
\beq\label{e:QIV2}
\|( I- \inter_h^{\rm V}) w\|_{L^2(K)} \leq Ch\sum_{K'\in \widetilde{\mathcal{T}}_K} 
\| w\|_{H^1(K')}
\eeq
(where $\widetilde{\mathcal{T}}_K$ is defined in Definition \ref{def:patch}).
\end{theorem}

Note that $C>0$ in Theorem \ref{t:QIV} depends on $c>0$ but not on $p_f, h$, and $k$.
}

\red{
\bpf[References for the proof]
We define $\inter_h^{\rm V}$ to be the operator $P_{\boldsymbol{hp}}$ from \cite[\S3]{Vo:25}. 
The property (i) is then, \cite[Equation 3.9]{Vo:25}
(ii) is \cite[Equation 3.14]{Vo:25}
and (iii) is \cite[Equation 3.15]{Vo:25}. 
\epf
}

\red{
\begin{corollary}\label{c:QIsupport}
For all $w\in H^1_0(\Omega)$, $\supp \,\inter_h^{\rm V} (\chi_\ell w)\subset \overline{\Omega_\ell}$ and thus $\inter_h^{\rm V} (\chi_\ell w)\in \subdomain$.
\end{corollary} 

\bpf
Recall from Assumption \ref{ass:subdomain_def} that $\Omega_\ell$ is a union of mesh elements, so it is sufficient to prove that $\|\inter_h^{\rm V}(\chi_\ell w)\|_{L^2(K)}=0$ for all $K\in \Omega_\ell^c$. Recall also from Assumption \ref{ass:subdomain_def} that if $K\in (\Omega_\ell)^c$, then $\widetilde{\mathcal{T}}_K\cap \supp\chi_\ell=\emptyset$. 
Now, for $K\in \Omega_\ell^c$, by the fact that $\supp \chi_\ell \subset \Omega_\ell$ and the bound \eqref{e:QIV2}, 
\beqs
\|\inter_h^{\rm V}(\chi_\ell w)\|_{L^2(K)}  \leq \|\chi_\ell w\|_{L^2(K)} +\|( I- \inter_h^{\rm V}) (\chi_\ell w)\|_{L^2(K)} 
\leq Ch\sum_{K'\in \widetilde{\mathcal{T}}_K} 
\| \chi_\ell w\|_{H^1(K')}=0.
\eeqs
\epf
}

\red{
\begin{theorem}\mythmname{A quasi-interpolant with optimal degree-explicit bounds  from \cite{KaMe:15}}
\label{t:QIKM}
Given $\Cshape>0$ there exists $C>0$ such that for $\coarse\subset H^1_0(\Omega)$ satisfying Assumption  \ref{ass:coarse_new} 
there exists a linear operator  $\inter_\Hcoarse^{\rm KM}: H^1_0(\Omega)\to \coarse$ such that for all $w\in H^1_0(\Omega)$,
\beq\label{e:QIKM1}
\| \nabla\big(( I- \inter_\Hcoarse^{\rm KM}) w)\|_{L^2(\Omega)} \leq C \| w\|_{H^1(\Omega)}
\eeq
and 
(b) 
\beq\label{e:QIKM2}
\|( I- \inter_\Hcoarse^{\rm KM}) w\|_{L^2(\Omega)} \leq C\frac{\Hcoarse}{p_c}
\| w\|_{H^1(\Omega)}.
\eeq
\end{theorem}
}

\red{
\bpf[References for the proof]
Set $\inter_\Hcoarse^{\rm KM}$ to be the operator $\ell^{hp}$ from 
\cite[Corollary 3.7]{KaMe:15}. 
The assumptions of \cite[Corollary 3.7]{KaMe:15} (which are the assumptions of \cite[Theorem 3.3]{KaMe:15}) are then satisfied with $\Pi^{\rm hp}$ the approximant from \cite[Appendix B]{MeSa:10}/\cite[Theorems A.11 and A.12]{GSAN}. 
The bounds \eqref{e:QIKM1} and \eqref{e:QIKM2} then follow from summing the second bound in \cite[Corollary 3.7]{KaMe:15} over elements of the mesh.
\epf

\bre\label{rem:QI}
The key differences between $\inter_h^{\rm V}$ and $\inter_\Hcoarse^{\rm KM}$ are that $\inter_\Hcoarse^{\rm KM}$ has the optimal $p$-dependence in the $L^2$ approximation result \eqref{e:QIKM2} whereas $\inter_h^{\rm V}$ does not \eqref{e:QIV2} (as noted in \cite[Remark 3.7]{Vo:25}), but $\inter_h^{\rm V}$ leaves the finite-element space invariant
but $\inter_\Hcoarse^{\rm KM}$ does not.

The fact that 
the right-hand side of \eqref{e:QIV2} does not have a divisor of $p_f$ (cf. \eqref{e:QIKM2}) 
is the reason why 
Informal Theorem \ref{thm:informal_pwp_add}/Theorem \ref{thm:final} involve fine spaces with dimension $\sim (kL)^d(\log (kL))^d$ instead of $\sim (kL)^d$. 
\ere
}

\subsection{Degree-explicit stable splittings}\label{sec:stable_splittings}
\red{
\begin{lemma}[Degree-independent, one-level stable splitting]\label{lem:split}
Given $A$ as in Assumption \ref{ass:sesqui} and $\Cshape,\Cpou,$
there exists $C>0$ such that for $\fine$ satisfying Assumption \ref{ass:fine} and subdomains satisfying Assumption \ref{ass:subdomain_def} the following holds. 
For all $p_f, N,\Lambda$, $h,k$,  and $v_h\in \fine$, there exist $v_{h,\ell} \in \subdomain$, $\ell=1,\ldots,N$, such that
\beq\label{eq:split_result}
v_h= \sum_{\ell=1}^N v_{h,\ell} \quad\tand\quad \sum_{\ell=1}^N \N{v_{h,\ell}}^2_{H^1_k(\Omega_\ell)} \leq C
(1+kh)^2
\Lambda \Big[\N{v_h}^2_{H^1_k(\Omega)}+ (k\delta)^{-2}\N{v_h}^2_{L^2(\Omega)}\Big].
\eeq
\end{lemma}

We call Lemma \ref{lem:split} a ``one-level" stable splitting, because the decomposition of $v_h$ in \eqref{eq:split_result} does not involve a component from the coarse space.

\bpf[Proof of Lemma \ref{lem:split}]
Let $v_{h,\ell}:= \inter_h^{\rm V} \big( \chi_\ell v_h\big)$ and observe that this is in $\subdomain$ 
by Corollary \ref{c:QIsupport}.
By \eqref{e:QIV1} and \eqref{e:QIV2} and the definition of $\|\cdot\|_{H^1_k}$ \eqref{eq:1knorm},
\beqs
\|(I-\inter_h^{\rm V})w\|_{H^1_k(K)}\leq C (1+ kh) \sum_{K'\in \widetilde{\mathcal{T}}_K}\|w\|_{H^1_k(K')},
\eeqs
so that 
\beq\label{e:QIV}
\big\|(I-\inter_h^{\rm V}) \big( \chi_\ell v_h\big)\big\|_{H^1_k(\Omega_\ell)}\leq C (1+ kh)\|\chi_\ell v_h\|_{H^1_k(\Omega_\ell)},
\eeq
where we have again used the support property of $\inter_h^{\rm V} ( \chi_\ell v_h)$.
By the product rule and \eqref{eq:PoU},
\begin{align*}
k^{-1}\N{\nabla (\chi_\ell v_h)}_{L^2(\Omega_\ell)} 
&\leq 
(k\delta_\ell)^{-1} \| v_h\|_{L^2(\Omega_\ell)} + k^{-1}\| \nabla v_h\|_{L^2(\Omega_\ell)}
\end{align*}
and thus 
\beq\label{e:trainLate}
\N{\chi_\ell v_{h}}_{H^1_k(\Omega_\ell)} \leq k^{-1}\| \nabla v_h\|_{L^2(\Omega_\ell)}+ (1+(k\delta_\ell)^{-1}) \| v_h\|_{L^2(\Omega_\ell)}.
\eeq
The combination of \eqref{e:QIV} and \eqref{e:trainLate} implies that
\begin{align*}
\N{v_{h,\ell}}_{H^1_k(\Omega_\ell)} 
&\leq \N{\chi_\ell v_{h}}_{H^1_k(\Omega_\ell)} + \big\|(I-\inter_h^{\rm V})(\chi_\ell v_h)\big\|_{H^1_k(\Omega_\ell)}\\
&\leq 
(1+kh)\Big(k^{-1}\| \nabla v_h\|_{L^2(\Omega_\ell)}+ (1+(k\delta_\ell)^{-1}) \| v_h\|_{L^2(\Omega_\ell)}\Big).
\end{align*}

Therefore, by the definition \eqref{eq:Hsub} of $\delta$,
\beqs
 \sum_{\ell=1}^N \N{v_{h,\ell}}^2_{H^1_k(\Omega_\ell)} \leq C
 (1+kh)
 ^2\bigg(  \sum_{\ell=1}^N\N{v_{h}}_{H^1_k(\Omega_\ell)}^2 +  \sum_{\ell=1}^N (k\delta)^{-2}\N{v_h}_{L^2(\Omega_\ell)}^2\bigg);
 \eeqs
the result \eqref{eq:split_result} then follows by applying \eqref{eq:overlap2a} with $v=w=v_h$.
\epf

\begin{lemma}[Degree-explicit, two-level stable splitting]\label{lem:split_two}
Given $\Cshape,\Cpou,$ and $A$ as in Assumption \ref{ass:sesqui}, there exists $C>0$ such that for  
$\fine$ satisfying Assumption \ref{ass:fine}  and subdomains satisfying Assumption \ref{ass:subdomain_def} the following holds.
For all $p_f, p_c,N,\Lambda$, $h,\Hcoarse,k$,  and $v_h\in \fine$, there exist $v_{h,0} \in \coarse$ and $v_{h,\ell} \in \subdomain$, $\ell=1,\ldots,N$, such that
\beq\label{eq:split_result_new}
v_h= \sum_{\ell=0}^N v_{h,\ell} \,\tand\, \sum_{\ell=0}^N \N{v_{h,\ell}}^2_{H^1_k(\Omega)} \leq C 
(1+kh)^2\Lambda \bigg[
\bigg(1 +\frac{k\Hcoarse}{p_c}\bigg)^2+ (k\delta)^{-2} \bigg(\frac{k\Hcoarse}{p_c}\bigg)^2\bigg] 
\N{v_h}^2_{H^1_k(\Omega)}.
\eeq
\end{lemma}
\bre
The classic stable splitting for two-level methods for Helmholtz problems (converted to working in the $H^1_k(\Omega)$ norm) is the following:~given $v_h\in \fine$, there exist $v_{h,0}\in \coarse$ and $v_{h,\ell}\in \subdomain, \ell=1,\ldots,N$, such that
\beq\label{eq:ss_two}
v_h= \sum_{\ell=0}^N v_{h,\ell} \quad\tand\quad \sum_{\ell=0}^N \N{v_{h,\ell}}^2_{H^1_k(\Omega)} \leq C \Lambda \bigg(1 + \frac{\Hcoarse}{\delta}\bigg)\N{v}_{H^1_k(\Omega)}^2,
\eeq
and $C$ depends in an unknown way
on the polynomial degree; see \cite[Lemma 4.1]{GSV1}.
We therefore see that the bound \eqref{eq:split_result_new} is supoptimal in its dependence on $\Hcoarse/\delta$, but explicit in $p_c$ and $p_f$.
\ere

\bpf[Proof of Lemma \ref{lem:split_two}]
Given $v_h\in \fine\subset H^1_0(\Omega)$, let 
\beqs
v_{h,0}:= \inter^{\rm KM}_\Hcoarse v_h \subset \coarse \subset \fine,
\eeqs
where $\inter^{\rm KM}_\Hcoarse$ is as in Theorem \ref{t:QIKM}. 
The bounds \eqref{e:QIKM1} and \eqref{e:QIKM2} and the definition \eqref{eq:1knorm} of $\|\cdot\|_{H^1_k(\Omega)}$
imply that 
\beq\label{eq:combine2new}
\N{ v_h-v_{h,0}}_{L^2(\Omega)}\leq C \frac{k\Hcoarse}{p_c} \N{v_h}_{H^1_k(\Omega)} \quad\tand \quad \N{ v_h-v_{h,0}}_{H^1_k(\Omega)}\leq C\bigg(1+\frac{k\Hcoarse}{p_c}\bigg)\N{v_h}_{H^1_k(\Omega)},
\eeq
where $C$ is independent of $k,h$, and $p_c$.
We now apply Lemma \ref{lem:split} to $v_h-v_{h,0}$; i.e., there exist $v_{h,\ell} \in \subdomain$, $\ell=1,\ldots, N$, such that 
$v_h-v_{h,0} = \sum_{\ell=1}^N v_{h,\ell}$
and
\beq\label{eq:combine3new}
\sum_{\ell=1}^N \N{v_{h,\ell}}^2_{H^1_k(\Omega_\ell)} \leq C 
(1+kh)^2\Lambda \Big[\N{v_h-v_{h,0}}^2_{H^1_k(\Omega)}+ (k\delta)^{-2}\N{v_h-v_{h,0}}^2_{L^2(\Omega)}\Big].
\eeq
The result \eqref{eq:split_result_new} then follows by inserting \eqref{eq:combine2new} into \eqref{eq:combine3new}.
\epf
}

The way the two-level stable splitting enters the proof of Theorem \ref{thm:main_add} is via the  following corollary, which combines Lemma \ref{lem:split_two} and the G\aa rding inequality \eqref{eq:Garding}.

\begin{lemma}[Two-level stable splitting combined with G\aa rding inequality]\label{lem:key_Garding_new}
\red{Let $a(\cdot,\cdot)$ satisfy Assumption \ref{ass:sesqui}.}
Given $\Cshape,\Cpou$  there exists $C>0$ 
such that for $\fine$ satisfying Assumption \ref{ass:fine} and subdomains satisfying Assumption \ref{ass:subdomain_def}, 
for all $p_f,p_c, N,\Lambda$, $h,\Hcoarse,k$, and $v_h\in \fine$, 
\beq\label{eq:key_Garding_new}
\N{v_h}^2_{H^1_k(\Omega)} \leq C \Lambda
\red{(1+kh)^2 \bigg[\bigg(
1 +\frac{k\Hcoarse}{p_c}\bigg)^2+ (k\delta)^{-2} \bigg(\frac{k\Hcoarse}{p_c}\bigg)^2\bigg] 
}
\sum_{\ell=0}^N \N{Q_\ell v_h}^2_{H^1_k(\Omega)}
+ C \N{v_h}^2_{L^2(\Omega)}.
\eeq
\end{lemma}

\bpf
By Lemma \ref{lem:split_two}, there exists $v_{h,\ell} \in \subdomain, \, \ell=0,\ldots, N$, such that $v_h = \sum_{\ell=0}^N v_{h,\ell}$. Then, by the G\aa rding inequality \eqref{eq:Garding} and the definition  \eqref{eq:Q} of $Q_\ell$,
\begin{align}\nonumber
c_{\rm G} \N{v_h}^2_{H^1_k(\Omega)} - C_{\rm G}\N{v_h}^2_{L^2(\Omega)} \leq \Re a(v_h,v_h) 
 &= \Re a\bigg(v_h, \sum_{\ell=0}^N v_{h,\ell}\bigg)\\ \nonumber
 &= \sum_{\ell=0}^N \Re a(v_h, v_{h,\ell})\\ 
&= \sum_{\ell=0}^N \Re a(Q_\ell v_h, v_{h,\ell})
\leq \sum_{\ell=0}^N \big|a(Q_\ell v_h, v_{h,\ell})\big|.
\label{eq:combine1}
\end{align}
By the bound \eqref{eq:cont2} (which holds since $v_{h,\ell}$ and $Q_\ell v_h\in \subdomain$), the Cauchy--Schwarz inequality, and Lemma \ref{lem:split_two},
\begin{align}\nonumber
&\sum_{\ell=0}^N \big|a(Q_\ell v_h, v_{h,\ell})\big|\leq \Ccont \sum_{\ell=0}^N \N{Q_\ell v_h}_{H^1_k(\Omega)} \N{v_{h,\ell}}_{H^1_k(\Omega_\ell)}\\ \nonumber
&\leq \Ccont\bigg( \sum_{\ell=0}^N \N{Q_\ell v_h}_{H^1_k(\Omega)}^2 \bigg)^{1/2} \bigg(\sum_{\ell=0}^N\N{v_{h,\ell}}_{H^1_k(\Omega_\ell)}^2\bigg)^{1/2}\\
&\leq C'\Ccont\bigg( \sum_{\ell=0}^N \N{Q_\ell v_h}_{H^1_k(\Omega)}^2 \bigg)^{1/2}
\Lambda^{1/2}
\red{(1+kh)\bigg[
\bigg(1 +\frac{k\Hcoarse}{p_c}\bigg)^2+ (k\delta)^{-2} \bigg(\frac{k\Hcoarse}{p_c}\bigg)^2\bigg]^{1/2}
}
 \N{v_h}_{H^1_k(\Omega)}.\label{eq:combine2}
\end{align}
Therefore, by combining \eqref{eq:combine1} and \eqref{eq:combine2} and using the inequality 
\beq\label{eq:PeterPaul}
ab \leq \epsilon a^2 + (4\epsilon)^{-1}b^2 \quad\tfa a,b,\epsilon>0,
\eeq
we obtain that 
\begin{align*}
& c_{\rm G} \N{v_h}^2_{H^1_k(\Omega)} - C_{\rm G}\N{v_h}^2_{L^2(\Omega)} \\
 &\leq 
\epsilon\N{v_h}_{H^1_k(\Omega)}^2+
 \epsilon^{-1} (C'\Ccont)^2\Lambda
\red{(1+kh)^2 \bigg[
\bigg(1 +\frac{k\Hcoarse}{p_c}\bigg)^2+ (k\delta)^{-2} \bigg(\frac{k\Hcoarse}{p_c}\bigg)^2\bigg] 
}
 \sum_{\ell=0}^N \N{Q_\ell v_h}^2_{H^1_k(\Omega)};
\end{align*}
the result \eqref{eq:key_Garding_new} then follows by taking, e.g., $\epsilon=c_{\rm G}/2$.
\epf

\subsection{
\red{Bounding certain $H^1_k$ inner products by weaker norms}}

We record the Poincar\'e inequality applied to the subdomains $\{\Omega_\ell\}_{\ell=1}^N$.

\begin{theorem}[Poincar\'e inequality applied to $\{\Omega_\ell\}_{\ell=1}^N$]\label{thm:Poincare}
\red{Given $A$ as in Assumption \ref{ass:sesqui}} there exists $C_P>0$ such that for all $k>0$ and $\ell \in \{1,\ldots,N\}$, 
\beq\label{eq:Poincare}
\N{v}_{L^2(\Omega_\ell)} \leq C_P\, k H_\ell \N{v}_{H^1_k(\Omega_\ell)} \quad\tfa v\in H^1_0(\Omega_\ell).
\eeq
\end{theorem}

\bpf
For domains of fixed size, the inequality
$\N{v}_{L^2(D)} \leq C \N{\nabla v}_{L^2(D)}$
is proved in, e.g., \cite[\S5.3]{BrSc:08}. 
A scaling argument then yields that $\N{v}_{L^2(D)} \leq C L \N{\nabla v}_{L^2(D)}$ for domains of characteristic length scale $L$, and then \eqref{eq:Poincare} follows from the definition \eqref{eq:1knorm} of $\|\cdot\|_{H^1_k(\Omega)}$.
\epf

\red{
\begin{lemma}[Bounding certain $H^1_k$ inner products by weaker norms]\label{lem:GSV}
Let $\sigma$ be as in Assumption \ref{ass:coarse}. 
There exists $C>0$ such that 
for all $k>0$ and $v_h\in \fine$, 
\beq\label{eq:L21}
\big|
\big( (I-Q_0)v_h, Q_0 v_h\big)_{H^1_k(\Omega)}
\big| 
\leq C \sigma \N{v_h}_{H^1_k(\Omega)} \N{Q_0 v_h}_{H^1_k(\Omega)}
\eeq
and 
\begin{align}
\big|
\big( (I-Q_\ell)v_h, Q_\ell v_h\big)_{H^1_k(\Omega)}
\big| 
&\leq C kH_\ell \N{(I-Q_\ell)v_h}_{H^1_k(\Omega_\ell)} \N{Q_\ell v_h}_{H^1_k(\Omega_\ell)}.
\label{eq:L22}
\end{align}
\end{lemma}

The key point is that the quantities $\sigma$ and $kH_\ell$ on the right-hand sides of \eqref{eq:L21} and \eqref{eq:L22} 
will be made sufficiently small (via imposing the condition \eqref{eq:newHsub}) 
in the course of the proof of the lower bound 
on the field of values of $Q$ \eqref{eq:upper_add2}.

\bpf[Proof of Lemma \ref{lem:GSV}]
By Assumption \ref{ass:sesqui}, the definition of $(\cdot,\cdot)_{H^1_k(\Omega)}$ \eqref{eq:1kip}, and integration by parts, for $u,v\in H^1_0(\Omega)$.
\begin{align}\label{e:chickenpox1}
a(u,v)- (u,v)_{H^1_k(\Omega)}
&=\int_\Omega \Big(k^{-1}( B \cdot \nabla u) \overline{v}+ (E-1) u \overline{v}\Big)\\
&=\int_\Omega \Big( -k^{-1}( B \cdot \overline{\nabla v}) u  - u\overline{v} (\nabla\cdot B) + (E-1) u \overline{v}\Big),\nonumber
\end{align}
so that 
\beq\label{e:theThing1}
\big|a(u,v)- (u,v)_{H^1_k(\Omega)}\big|\leq C \| u\|_{L^2(\Omega)}\|v\|_{H^1_k(\Omega))}
\eeq
where we have used  the assumptions that 
$B\in W^{1,\infty}(\Omega,\Com^d)$ and $E\in L^\infty(\Omega,\Com)$ with norms bounded above uniformly in $k$ (in Assumption \ref{ass:sesqui}). 

By  \eqref{eq:Q}, 
$a\big( (I-Q_0)v_h, Q_0 v_h\big)=0.$ Thus 
\beqs
\big|
\big( (I-Q_0)v_h, Q_0 v_h\big)_{H^1_k(\Omega)}\big| \leq C \| (I-Q_0)v_h\|_{L^2(\Omega)} \| Q_0 v_h\|_{H^1_k(\Omega)},
\eeqs
and the bound \eqref{eq:L21} then follows from the second bound in \eqref{eq:QO2a}. 

Similarly, by  \eqref{eq:Q}, 
$a( (I-Q_\ell)v_h, Q_\ell v_h)=0.$
Thus, by the fact that $Q_\ell v_h$ is supported in $\overline{\Omega_\ell}$ and \eqref{e:chickenpox1}, 
\begin{align*}
\big|
\big( (I-Q_\ell)v_h, Q_\ell v_h)_{H^1_k(\Omega)}
\big| 
&= \Big|\int_{\Omega} \Big(k^{-1}( B \cdot \nabla ((I-Q_\ell)v_h)) \overline{Q_\ell v_h}+ (E-1) \big((I-Q_\ell)v_h \big)\overline{Q_\ell v_h}\Big)\Big|\\
&= \Big|\int_{\Omega_\ell} \Big(k^{-1}( B \cdot \nabla ((I-Q_\ell)v_h)) \overline{Q_\ell v_h}+ (E-1) \big((I-Q_\ell)v_h \big)\overline{Q_\ell v_h}\Big)\Big|\\
&\leq C\|(I-Q_\ell)v_h\|_{H^1_k(\Omega_\ell)}\|Q_\ell v_h\|_{L^2(\Omega_\ell)}.
\end{align*}
The result \eqref{eq:L22} then follows from 
Theorem \ref{thm:Poincare} since $Q_\ell v_h\in H^1_0(\Omega_\ell)$.
\epf
}

\red{
\bre[The reason for the assumption that the coefficient $A$ is Hermitian]\label{rem:why}
The requirement  that $A$ is Hermitian (Assumption \ref{ass:sesqui}) is vital for the proof of Lemma \ref{lem:GSV}. 
Indeed, for \eqref{e:theThing1} to hold, we need to have the coefficient $A$ inside 
the inner-product $(\cdot,\cdot)_{H^1_k(\Omega)}$, and for this to indeed be an inner product, we need  $A$ to be  Hermitian.
\ere
}


\section{Proof of Theorem \ref{thm:main_add}}\label{sec:proof_abstract_add}

In this proof, $C$ and $C'$ denote quantities 
that may depend on
$\Cpou, \Cnormo, \Cnormell$, 
$k_0$ (i.e., the quantities specified at the start of the statement of Theorem \ref{thm:main_add}) and 
whose values may change from line to line in the proofs.

\subsection{Proof of the upper bound \eqref{eq:upper_add1}}

By \eqref{eq:QO2a} and the triangle inequality,
\beq\label{eq:Fridayrain2}
\N{Q_0v_h }_{H^1_k(\Omega)} \leq \big(1+\Cnormo\big)\N{v_h}_{H^1_k(\Omega)}
\quad\tfa v_h\in \fine.
\eeq
Furthermore, by \eqref{eq:overlap_new},  \eqref{eq:Qellbound}, and the second inequality in \eqref{eq:overlap2a}, 
\beqs
\bigg\|\sum_{\ell=1}^N Q_\ell v_h\bigg\|_{H^1_k(\Omega)}^2\leq 2\Lambda \sum_{\ell=1}^N \N{Q_\ell v_h}_{H^1_k(\Omega_\ell)}^2 \leq 
2\Lambda (C_{\rm sub})^2 \sum_{\ell=1}^N \N{v_h}^2_{H^1_k(\Omega_\ell)}
\leq 2\Lambda^2 (C_{\rm sub})^2 \N{v_h}^2_{H^1_k(\Omega)}.
\eeqs
Then, by the definition of $\Qa$ \eqref{eq:pc}, the triangle inequality, and \eqref{eq:PeterPaul},
\beqs
\N{\Qa v_h}_{H^1_k(\Omega)}^2 \leq 2\N{Q_0v_h}^2_{H^1_k(\Omega)} + 2\bigg\|\sum_{\ell=1}^N Q_\ell v_h\bigg\|^2_{H^1_k(\Omega)},
\eeqs
and the result follows.


\subsection{Proof of the lower bound \eqref{eq:upper_add2}}

\paragraph{Overview of the proof.}

By \eqref{eq:pc}, for all $v_h\in \fine$,
\begin{align}\label{eq:explain1}
\big( v_h, \Qa v_h\big)_{H^1_k(\Omega)} & = \sum_{\ell=0}^N \big( v_h, Q_\ell v_h\big)_{H^1_k(\Omega)}
=\sum_{\ell=0}^N \bigg[ \big\| Q_\ell v_h\big\|^2_{H^1_k(\Omega)} + \Big( (I-Q_\ell) v_h, Q_\ell v_h\Big)_{H^1_k(\Omega)}\bigg].
\end{align}
For $\Qa$, the strategy to obtain a lower bound on the field of values is to 
\ben
\item Use  Lemma \ref{lem:key_Garding_new} to bound $\sum_{\ell=0}^ N  \| Q_\ell v_h\|^2_{H^1_k(\Omega)}$ 
on the right-hand side of \eqref{eq:explain1} 
from below in terms of  $\|v_h\|^2_{H^1_k(\Omega)}$, and 
\item show that the ``cross terms'' $( (I-Q_\ell) v_h, Q_\ell v_h)_{H^1_k(\Omega)}$, $\ell=0,\ldots,N$, can be absorbed into the other (non-negative) terms using Lemma \ref{lem:GSV}.
\een

To keep the argument concise, we introduce the notation 
\beq\label{e:Css}
C_{\rm ss,2}:= C \Lambda
\red{(1+kh)^2 \bigg[
\bigg(1 +\frac{k\Hcoarse}{p_c}\bigg)^2+ (k\delta)^{-2} \bigg(\frac{k\Hcoarse}{p_c}\bigg)^2\bigg],
}
\eeq
with $C$ as in Lemma  \ref{lem:key_Garding_new},
so that 
\eqref{eq:key_Garding_new} becomes
\beq\label{eq:key_Garding_two}
\N{v_h}^2_{H^1_k(\Omega)} \leq C_{\rm ss,2}
\sum_{\ell=0}^N \N{Q_\ell v_h}^2_{H^1_k(\Omega)}
+ C\N{v_h}^2_{L^2(\Omega)}.
\eeq
By the second bound in \eqref{eq:QO2a}, $\|v_h\|_{L^2(\Omega)}\leq \|Q_0 v_h\|_{L^2(\Omega)} + \sigma \|v_h\|_{H^1_k(\Omega)}$, 
so that 
\beqs
\|v_h\|_{L^2(\Omega)}^2\leq 2\|Q_0 v_h\|_{L^2(\Omega)}^2 +2\sigma^2  \|v_h\|_{H^1_k(\Omega)}^2,
\eeqs
and combining this with \eqref{eq:key_Garding_two}, we obtain that 
\beq\label{eq:key_Garding_two_alt}
(1- 2C \sigma^2)\N{v_h}^2_{H^1_k(\Omega)} \leq \big(C_{\rm ss,2} +2C\big)
\sum_{\ell=0}^N \N{Q_\ell v_h}^2_{H^1_k(\Omega)},
\eeq
which directly achieves the goal in Step 1 above if $\sigma$ is sufficiently small. 

Now, by \eqref{eq:explain1}, 
\begin{align*}
\big( v_h, \Qa v_h\big)_{H^1_k(\Omega)} & =\frac12\sum_{\ell=0}^N  \big\| Q_\ell v_h\big\|^2_{H^1_k(\Omega)} 
+\frac12\sum_{\ell=0}^N \big\| Q_\ell v_h\big\|^2_{H^1_k(\Omega)} 
+\sum_{\ell=0}^N \Big( (I-Q_\ell) v_h, Q_\ell v_h\Big)_{H^1_k(\Omega)},
\end{align*}
and then, by \eqref{eq:key_Garding_two_alt},
\beqs
\big|\big( v_h, \Qa v_h\big)_{H^1_k(\Omega)}\big|\geq \frac{1- 2C \sigma^2}{2 (C_{\rm ss,2} + 2C)} \N{v_h}^2_{H^1_k(\Omega)} + \frac12 \sum_{\ell=0}^N \N{Q_\ell v_h}^2_{H^1_k(\Omega)}
+\sum_{\ell=0}^N
\Big( (I-Q_\ell) v_h, Q_\ell v_h\Big)_{H^1_k(\Omega)}.
\eeqs
We now deal with the cross terms. For the term with $\ell=0$, by \eqref{eq:L21} and \eqref{eq:PeterPaul},
\begin{align*}\nonumber
\big|\big((I-Q_0)v_h, Q_0 v_h\big)_{H^1_k(\Omega)} \big|
&\leq C \sigma \N{v_h}_{H^1_k(\Omega)} \N{Q_0 v_h}_{H^1_k(\Omega)}\leq \frac{(C\sigma)^2 }{4\epsilon} \N{v_h}_{H^1_k(\Omega)}^2 + \epsilon \N{Q_0 v_h}_{H^1_k(\Omega)}^2,
\end{align*}
so that 
\begin{align}\nonumber
\big|\big( v_h, \Qa v_h\big)_{H^1_k(\Omega)}\big|
&\geq \frac{1- 2C \sigma^2}{2 (C_{\rm ss,2} + 2C)} \N{v_h}^2_{H^1_k(\Omega)} + \frac12 \sum_{\ell=0}^N \N{Q_\ell v_h}^2_{H^1_k(\Omega)}\\
&\qquad 
- \epsilon \N{Q_0 v_h}^2_{H^1_k(\Omega)} - C \epsilon^{-1} \sigma^2 \N{v_h}^2_{H^1_k(\Omega)}
+\sum_{\ell=1}^N
\Big( (I-Q_\ell) v_h, Q_\ell v_h\Big)_{H^1_k(\Omega_\ell)}.\label{eq:lastday0}
\end{align}
For the cross terms with $\ell=1,\ldots,N$, by \eqref{eq:L22} and \eqref{eq:Qellbound},
\beqs
\big|
\big((I-Q_\ell)v_h, Q_\ell v_h\big)_{H^1_k(\Omega)} 
\big| 
\leq C kH_\ell(1+ C_{\rm sub}) \|v_h\|_{H^1_k(\Omega_\ell)}  \N{Q_\ell v_h}_{H^1_k(\Omega_\ell)}.
\eeqs
Then, by Cauchy--Schwarz, the second inequality in \eqref{eq:overlap2a}, and \eqref{eq:PeterPaul} again,
\begin{align}\nonumber
\Big|
\sum_{\ell=1}^N
\Big( (I-Q_\ell) v_h, Q_\ell v_h\Big)_{H^1_k(\Omega)}\Big|
&\leq
C'\sum_{\ell=1}^N kH_\ell  \N{v_h}_{H^1_k(\Omega_\ell)}\N{Q_\ell v_h}_{H^1_k(\Omega_\ell)}\\ \nonumber
&\leq C' k\Hsub \bigg(\sum_{\ell=1}^N  \N{v_h}_{H^1_k(\Omega_\ell)}^2 \bigg)^{1/2} \bigg( \sum_{\ell=1}^N \N{Q_\ell v_h}_{H^1_k(\Omega_\ell)}^2\bigg)^{1/2} \\ \nonumber
&\leq C' k\Hsub \Lambda^{1/2}\N{v_h}_{H^1_k(\Omega)}\bigg( \sum_{\ell=1}^N \N{Q_\ell v_h}_{H^1_k(\Omega_\ell)}^2\bigg)^{1/2} \\
&\leq \epsilon  \sum_{\ell=1}^N \N{Q_\ell v_h}_{H^1_k(\Omega_\ell)}^2 + C' \epsilon^{-1} (k\Hsub)^2 \Lambda \N{v_h}^2_{H^1_k(\Omega)}\label{eq:lastday1}.
\end{align}
Inserting \eqref{eq:lastday1} into \eqref{eq:lastday0}, choosing $\epsilon=1/2$, and using that $\N{Q_\ell v_h}_{H^1_k(\Omega_\ell)}= \N{Q_\ell v_h}_{H^1_k(\Omega)}$ (since $Q_\ell v_h$ is supported in $\overline{\Omega_\ell}$) we obtain that 
\beqs
\big|\big( v_h, \Qa v_h\big)_{H^1_k(\Omega)}\big|\geq 
\bigg[
\frac{1- 2C\sigma^2}{2 (C_{\rm ss,2} + 2C)} - 2C' \sigma^2 - 2C' (k\Hsub)^2 \Lambda
\bigg]\N{v_h}^2_{H^1_k(\Omega)}.
\eeqs
The lower bound on the field of values then follows if 
\red{
\beqs
\sigma^2, \quad \sigma^2 \big(C_{\rm ss,2}+2C\big), \quad \tand\quad (k\Hsub)^2 \Lambda \big(C_{\rm ss,2}+2C\big)
\eeqs
are all sufficiently small, with these conditions ensured by \eqref{eq:newHsub} (since $C_{\rm ss,2}$ is given by \eqref{e:Css}).}

\section{Results about the Helmholtz CAP problem}
\label{sec:CAP}

\red{This section focuses on the Helmholtz CAP problem of Definition \ref{def:CAP}; recall from Example \ref{ex:CAP} that} the sesquilinear form \eqref{eq:sesqui_CAP} of this problem satisfies Assumption \ref{ass:sesqui}. \red{We then define $(\cdot,\cdot)_{H^1_k(\Omega)}$ by \eqref{eq:1kip} with $A=A_{\rm scat}$.}
\red{We first show that problem \eqref{eq:CAP} is well-posed at the PDE level and then prove convergence for the $hp$-FEM method applied to this problem.}
\subsection{Results on the PDE level}

\begin{theorem}\label{thm:existence}
The solution of the CAP problem of Definition \ref{def:CAP} exists and is unique.
\end{theorem}

\bpf
Since $a(\cdot,\cdot)$ satisfies the G\aa rding inequality \eqref{eq:Garding}, by Fredholm theory (see, e.g., \cite[Theorem 2.33]{Mc:00}) it is sufficient to prove uniqueness.
By taking $v=u$ in the variational problem $a(u,v)=(f,v)_{L^2(\Omega)}$ and then 
taking the imaginary part, 
we see that if $f=0$, then $u= 0$ on $\supp V$. Thus $u= 0$ on $\Omega$ by the unique continuation principle (see 
\cite{Ca:39}
for the case $d=2$ and \cite{Ho:59}
for the case $d=3$). 
\epf

The next two results (Theorems \ref{thm:nt_CAP} and \ref{thm:CAP}) motivate the use of the CAP problem as an approximation to the Helmholtz scattering problem, but are not used in the rest of the paper. The proofs of these results are therefore relegated to Appendix \ref{app:CAP}.

With $u$ the solution of the CAP problem of Definition \ref{def:CAP}, let 
\beq\label{eq:Csol}
\Csol:= \sup_{f\in (H_0^1(\Omega))^*} \frac{ \N{u}_{H^1_k(\Omega)}}{\N{f}_{(H^1_k(\Omega))^*}}\quad\text{ where } \quad 
\N{f}_{(H^1_k(\Omega))^*}:= \sup_{v\in H^1_0(\Omega)}\frac{
\big|\langle f, v\rangle_{(H^1_0(\Omega))^*\times H^1_0(\Omega)}\big|
}{
\N{v}_{H^1_k(\Omega)}
}.
\eeq
 
\begin{theorem}\mythmname{\red{For nontrapping problems, the CAP problem inherits the bound on the solution operator of the scattering problem}}
\label{thm:nt_CAP}
Suppose that $A_{\rm scat}$ and $c_{\rm scat}$  are as in Definition \ref{def:scattering} and $\Omega_{\rm int}$ and $\Omega$ are as in Definition \ref{def:CAP}. 
Suppose that $\Omega$ has characteristic length scale $L$. 
Suppose that $A_{\rm scat}$ and $c_{\rm scat}$ are $C^\infty$ and nontrapping (in the sense of, e.g., \cite[Definition 4.42]{DyZw:19}). Then, 
given $k_0>0$, there exists $C>0$ such that, for all $k\geq k_0$,
$\Csol \leq CkL.$
\end{theorem}

\begin{theorem}\textbf{\emph{(The error caused by  CAP's approximation of the radiation condition is smooth and superalgebraically small in $k$)}}
\label{thm:CAP}
Given $A_{\rm scat}, c_{\rm scat}, \Omega_{\rm int}, \Omega$, and $V$ as in Definition \ref{def:CAP}, suppose that \emph{either} the Helmholtz scattering problem of Definition \ref{def:scattering}
is nontrapping \emph{or} 
both $\Csol$ and the solution operator of \red{the scattering problem}
are polynomially bounded in $kL$, where $L$ is the characteristic length scale of $\Omega$.  
Then, for all $k_0,s, M>0$ and $\chi \in C^\infty_{\rm comp}(\Omega_{\rm int})$, there exists $C>0$ such that the following is true for all $k\geq k_0$. 

Given $f\in (H^1_0(\Omega))^*$ with $\supp f \subset \overline{\Omega_{\rm int}}$, let $v$ be the solution of the Helmholtz scattering problem (as in Definition \ref{def:scattering}) and let $u$ be the solution of the CAP problem (as in Definition \ref{def:CAP}). 
Then 
\beq\label{eq:small}
\sum_{|\alpha|\leq s} k^{-s}\big\|\partial^\alpha \big(\chi(u-v))\big\|_{L^2(\Omega)} \leq C (kL)^{-M} \N{f}_{
(H^1_k(\Omega))^*
}.
\eeq
\end{theorem}

\red{
\bre[Comparison between CAP and PML]\label{rem:CAPvsPML}
Another way of approximating the radiation condition using complex absorption is perfectly matched layer (PML) truncation. 
The advantages of PML truncation over CAP are that (i) in some cases, PML convergence can be proved to be exponential in $k$ \cite{GLS2} (not just superalgebraic), 
and (ii) existing PML error bounds give control on how the error decreases as a function of both the PML width and the strength of the PML scaling (see \cite{GLS2} and the references therein).

The advantage of CAP over PML is that the second-order term in the CAP PDE \eqref{eq:CAP} is self-adjoint (which is not the case for PML); in other words, the non-self-adjointness coming from approximating the (non-self-adjoint) radiation condition is contained only in the lowest-order term of the PDE. This implies that (i) CAP fits into the framework of this paper, and (ii) 
the optimal $k$-explicit 
quasi-optimality results for both the $h$-FEM and the $hp$-FEM can be proved 
for the CAP problem when $\Omega$ is only assumed to be a convex polygon/polyhedron (see \S\ref{sec:hpFEM} below), whereas the current best theory for PML requires $\partial \Omega$ to be $C^{1,1}$ (see \cite[\S6-7]{GSAN}).
%
\ere
}
%

\subsection{$hp$-FEM convergence result}\label{sec:hpFEM}

\red{
\begin{theorem}\mythmname{Bound on ``adjoint approximablity factor'' 
for the CAP problem on convex polyhedra}\label{thm:hpFEM}
Suppose that $\Omega, A_{\rm scat}, $ $c_{\rm scat}$, $\Omega_{\rm int}$, and $V$ are as in Definition \ref{def:CAP} and, furthermore, that $\Omega$ is convex.
Let $\cS^* : (H^1_0(\Omega))^*\to H^1_0(\Omega)$ denote the adjoint of the CAP solution operator.

Given $k_0,\Cshape, M,N>0$ there exists $\mathcal{C}, h_0>0$ such that if $\cV_h$ satisfies Assumption \ref{ass:fine} with shape-regularity constant $\Cshape>0$, then for all $p_f\geq 1,0<h\leq h_0$, and $k\geq k_0$,
\begin{align}
&\| (I-\Pi_h)\mathcal{S}^*\|_{L^2(\Omega)\to H^1_k(\Omega)}
\leq \mathcal{C}\Bigg[\frac{kh}{p_f}+
\Big(\frac{kh}{p_f}\Big)^Mk^{-N} 
\Csol
+
\frac{(\mathcal{C}h)^{p_f}}{k}
\max\Big\{\frac{k}{p_f+1},1\Big\}^{p_f+1}
\Csol
\Bigg]
\label{e:etaUs}
\end{align}
where $\Pi_h: H^1_0(\Omega) \to \fine$ denotes the orthogonal projection.
\end{theorem}

Theorem \ref{thm:hpFEM} is proved below; combining this result with the Schatz argument (see, e.g., \cite[\S4]{MeSa:10}) under the following assumption, we obtain the $hp$-FEM convergence result of Corollary \ref{cor:hpFEM}.

\begin{assumption}[Polynomial boundedness of the CAP solution operator]\label{ass:poly}
$K\subset[k_0,\infty)$ is such that there exists $C,P>0$ such that $\Csol\leq C(kL)^\polyexp$ for $k\in K$.
\end{assumption}

By Theorem \ref{thm:nt_CAP}, Assumption \ref{ass:poly} holds at least when the scattering problem is nontrapping.

\begin{corollary}[$hp$-FEM convergence 
for the CAP problem on convex polyhedra]\label{cor:hpFEM}
Suppose that $\Omega, A_{\rm scat}, $ $c_{\rm scat}$, $\Omega_{\rm int}$, and $V$ are as in Definition \ref{def:CAP}, $\Omega$ is convex, and Assumption \ref{ass:poly} holds with a set $K$.
Given $\Cdegree$ and $k_0$ and for all $\epsilon>0$, there exist $C_1, C_2$ such that the following is true.
Given $f\in L^2(\Omega)$, if $k\geq k_0$ then the solution $u\in H^1_0(\Omega)$ of the CAP problem of Definition \ref{def:CAP} exists and is unique. Furthermore, if $\cV_h$ satisfies Assumption \ref{ass:fine} with shape-regularity constant $\Cshape$, 
$k \in K\cap[k_0,\infty)$, 
\beq\label{eq:hpfem1}
\frac{kh}{p_f} \leq C_1, \quad\tand\quad p_f\geq 1+\Cdegree \log k,
\eeq
then the Galerkin approximation $u_h$ to $u$ in $\fine$ exists, is unique, and satisfies 
\begin{gather}\label{eq:H1bound2new}
\N{u-u_h}_{H^1_k(\Omega)}\leq 2 \Ccont \min_{v_h \in \fdspace} \N{u-v_h}_{H^1_k(\Omega)}
\quad\tand\quad
\N{u-u_h}_{L^2(\Omega)}\leq \epsilon
\N{u-v_h}_{H^1_k(\Omega)}.
\end{gather}
Furthermore, the following bound on the discrete inf-sup constant holds
\beq\label{eq:dis}
\inf_{u_h\in \fine}\sup_{v_h\in \fine} \frac{
|a(u_h,v_h)|
}{
\N{u_h}_{H^1_k(\Omega)}\N{v_h}_{H^1_k(\Omega)}
}
\geq (C_2 \Csol)^{-1}.
\eeq
\end{corollary}

In the rest of the paper, Theorem \ref{thm:hpFEM} and Corollary \ref{cor:hpFEM} are used in two ways:~(i) to show that the Galerkin solution of the CAP problem \eqref{eq:sesqui_CAP} in the fine space exists, is unique, and is $k$-uniformly quasi-optimal, and (ii) to show that the coarse space $\coarse$ satisfies Assumption \ref{ass:coarse} with $\sigma$ sufficiently small (all under suitable assumptions on the fine and coarse meshwidths and polynomial degrees).
}

The proof of Theorem \ref{thm:hpFEM} requires the following lemma (proved by integrating by parts).

\ble[Bound on the solution of the CAP problem \red{in the CAP region}]\label{lem:ibps}
Suppose that $\Omega,\Omega_{\rm int}$, and $V$ are as in Definition \ref{def:CAP}. Given $a>0$ there exists $C>0$ such that the following is true. Suppose that $\phi \in C^\infty(\Rea^d;\Rea)$ is such that $\supp\, \phi \cap \Omega\subset \{ x:V(x)\geq a>0\}$. 
Introduce  the operator $\opP:= -k^{-2}\Delta - 1- \ri V$. 
If $u\in H^1_0(\Omega)$, then 
\beq\label{eq:ibps}
\N{\phi u}_{H^1_k(\Omega)} \leq C
\N{ \opP(\phi u)}_{(H^1_k(\Omega))^*}.
\eeq
\ele

\bpf
By integrating by parts/Green's identity, since $u=0$ on $\partial\Omega$,
\beq\label{eq:ibps1}
\big\langle \opP(\phi u),\phi u\big\rangle_{\red{(H^1_0(\Omega))^*\times H^1_0(\Omega)}} =k^{-2} \N{\nabla(\phi u)}^2_{L^2(\Omega)} - \N{\phi u}^2_{L^2(\Omega)} - \ri \big\|V^{1/2}\phi u\big\|^2_{L^2(\Omega)}.
\eeq
Taking the imaginary part of \eqref{eq:ibps1} and using the assumption on  $\supp\, \phi \cap \Omega$, we obtain that
\beqs
a \N{\phi u}^2_{L^2(\Omega)}
\leq  \big\|V^{1/2}\phi u\big\|^2_{L^2(\Omega)}
 = -\Im \big\langle \opP(\phi u),\phi u\big\rangle_{\red{(H^1_0(\Omega))^*\times H^1_0(\Omega)}}\leq \N{\opP(\phi u)}_{(H^1_k(\Omega))^*} \N{\phi u}_{H^1_k(\Omega)}
\eeqs
so that, by \eqref{eq:PeterPaul}, for all $\epsilon>0$, 
\beq\label{eq:ibps2}
2a\N{\phi u}_{L^2(\Omega)}\leq \epsilon^{-1}\N{\opP(\phi u)}_{(H^1_k(\Omega))^*}^2 + \epsilon \N{\phi u}^2_{H^1_k(\Omega)}.
\eeq
Taking the real part of \eqref{eq:ibps1}, we also  obtain that
\beqs
\N{\phi u}_{H^1_k(\Omega)}^2:= k^{-2} \N{\nabla(\phi u)}^2_{L^2(\Omega)}+ \N{\phi u}^2_{L^2(\Omega)} \leq \N{\opP(\phi u)}_{(H^1_k(\Omega))^*} \N{\phi u}_{H^1_k(\Omega)}+2\N{\phi u}^2_{L^2(\Omega)} ,
\eeqs
so that, by \eqref{eq:PeterPaul} again, 
\beq\label{eq:ibps3}
\N{\phi u}_{H^1_k(\Omega)}^2\leq C \Big( \N{\opP(\phi u)}_{(H^1_k(\Omega))^*}^2+\N{\phi u}^2_{L^2(\Omega)}\Big).
\eeq
The result \eqref{eq:ibps} then follows from the combination of \eqref{eq:ibps2} and \eqref{eq:ibps3}.
\epf

\red{
\bpf[References for the proof of Theorem \ref{thm:hpFEM}]
\cite[Theorem 7.2]{GSAN} proves the analogue of Theorem \ref{thm:hpFEM} for the Helmholtz radial PML problem. (Note that in Theorem \ref{thm:hpFEM} there is no Dirichlet obstacle, and so $\Gamma_-$ in \cite[Theorem 7.2]{GSAN} is the empty set.) 
\cite[Theorem 7.2]{GSAN} is proved via the Schatz argument; i.e., showing that the ``adjoint approximability" factor
$\| (I-\Pi_h) \mathcal{S}^*\|_{L^2(\Omega)\to H^1_k(\Omega)}$ can be made arbitrarily small under the conditions \eqref{eq:hpfem1}. This result is proved by combining ideas from \cite{LSW3, LSW4, GLSW1, GS3}, and splitting the solution operator into components of frequency $\lesssim k$ and components of frequency $\gg k$ (these ``frequency splitting" ideas were first introduced in the analysis of the $hp$-FEM in \cite{MeSa:10, MeSa:11}).

The result for the CAP problem (as opposed to the radial PML problem) follows by making the following two small changes  in the arguments in \cite[\S7]{GSAN}. 

(i) One ingredient to the proof of \cite[Theorem 7.2]{GSAN} is the bound \cite[Equation 7.24]{GSAN} showing that the restriction to the PML region of the solution of the Helmholtz PML problem is bounded uniformly in $k$ in terms of the data. 
This is proved in \cite[Lemma C.5]{AGS2}. Replacing \cite[Equation C.2]{AGS2} by \eqref{eq:ibps} establishes the analogous result for the Helmholtz CAP problem.

(ii) \cite[Theorem 7.2]{GSAN} requires the truncation boundary $\partial \Omega$  to be $C^{1,1}$ -- this is to ensure that the Helmholtz PML solution is $H^2$ in a neighbourhood of this boundary.
Because $A_{\rm scat}$ is real-valued (even in the CAP region), the CAP problem only requires that $\Omega$ be convex to have this regularity;
see, e.g., \cite[\S8.2 and Equation 8.2.2]{Gr:85}.
\epf
}

\subsection{\red{Discussion of the novelty of Theorem \ref{thm:hpFEM} and Corollary \ref{cor:hpFEM} in comparison to the analogous results about the interior impedance problem}}\label{sec:other_people}

For the CAP problem on a convex  polygon/polyhedron, Corollary \ref{cor:hpFEM} proves that the Galerkin solution exists, is unique, and is quasi-optimal, when $k,h$, and $p_f$ satisfy \eqref{eq:hpfem1}. 

In the simpler setting of fixed $p_f$,  Theorem \ref{thm:hpFEM} combined with the Schatz argument proves the analogous result for the $h$-FEM under the familiar condition ``$(kh)^{p_f}\Csol$ sufficiently small".
In contrast, there do not yet exist analogous results for the interior impedance problem on convex polygons/polyhedra, and we now give an explanation of why this is the case.

\paragraph{\red{Boundary regularity requirements  to prove well-posedness of the Galerkin solution for the interior impedance problem.}}

When $p_f=1$, the duality arguments used to prove well-posedness of the Galerkin solution for the interior impedance problem  \cite{MeSa:10, MeSa:11, DuWu:15, ChNi:20, GS3} 
 hold when $\Omega$ is a convex polygon/polyhedron and $A=I$.
However, for fixed $p_f>1$, these duality arguments 
all need $\partial\Omega$ to be smoother than Lipschitz, with the following exceptions:
\bit
\item \red{the results of \cite{ChNi:20} that allow domains with corners (but not edges) and any $p_f\geq 1$, but rqeuire the mesh to be refined in a specific way (see \cite[Equations 2.27 and 2.28]{ChNi:20}), and}
\item the results of \cite{EsMe:12} that prove the analogue of Corollary \ref{cor:hpFEM} for the $hp$-FEM applied to the interior impedance problem in a convex (2-d) polygon, under appropriate mesh refinement at the corners.
\eit

\paragraph{\red{Relevance for two-level DD analyses of the interior impedance problem.}}
For the set up considered in \cite{HuLi:24, LuXuZhZo:24, MaAlSc:24, FuGoLiWa:24} --  namely the interior impedance problem in a Lipschitz polygon/polyhedron solved using the FEM with fixed $p_f$ -- 
the only fine spaces currently proved to satisfy the \red{well-posedness} assumptions in \cite{HuLi:24, LuXuZhZo:24, MaAlSc:24, FuGoLiWa:24} 
\footnote{
\red{These assumptions are phrased in terms of the discrete inf-sup constant; the fact that the discrete inf-sup constant inherits the corresponding bound on the (continuous) inf-sup constant was proved in the asymptotic regime in \cite[Theorem 4.2]{MeSa:10} and recently in the preasymptotic regime in \cite[Corollary 5.11]{Sp:25}.}
}
have, in 3-d, at least $(kL)^{d}(\Csol)^{d/2}\gtrsim (kL)^{3d/2}$ degrees of freedom (since $\Csol \gtrsim kL$), regardless of polynomial degree (where $(kL)^{d}(\Csol)^{d/2}$ is the number of degrees of freedom when $(kh)^{2}\Csol \sim1$). Fine spaces with fewer degrees of freedom are allowed in 2-d, provided that the mesh is refined appropriately towards the corners.

\paragraph{\red{Why the requirements for well-posedness of the Galerkin solution are weaker for CAP than for the interior impedance problem.}}
\red{As stated above, Theorem \ref{thm:hpFEM} and Corollary \ref{cor:hpFEM} prove well-posedness of the Galerkin solution for the CAP problem on a convex  polygon/polyhedron.}
The key point is that the proof of Theorem \ref{thm:hpFEM} 
uses the ``good'' behaviour of the Helmholtz solution operator in the CAP region (as shown in Lemma \ref{lem:ibps}) to only require $H^2$ regularity of the solution there; 
since the leading-order term of the CAP PDE is the Laplacian, $H^2$ regularity holds when the CAP boundary is a convex polygon/polyhedron. 

With Theorem \ref{thm:hpFEM} in hand, the analyses in \cite{HuLi:24, LuXuZhZo:24, MaAlSc:24, FuGoLiWa:24} can, in principle, be repeated for the CAP problem and give results with 
$\ll (kL)^{3d/2}$ degrees of freedom in the fine space. 

\section{Rigorous statement of  Informal Theorem \ref{thm:informal_pwp_add}}\label{sec:pwp}

The conclusion of Theorem \ref{thm:final} is that the bounds on $Q$ in \eqref{eq:upper_add1} and \eqref{eq:upper_add2} hold. For brevity, the corollaries about GMRES 
are not explicitly written out, but follow from Corollaries \ref{cor:1}, \ref{cor:2}, and \ref{cor:3}.

\red{
  \begin{theorem}[Rigorous statement of Informal Theorem \ref{thm:informal_pwp_add}]
    \label{thm:final}
Suppose that $\Omega, A_{\rm scat}, $ $c_{\rm scat}$, $\Omega_{\rm int}$, and $V$ are as in Definition \ref{def:CAP}.
Suppose, additionally, that $\Omega$ is convex and Assumption \ref{ass:poly} holds with a set $K$.
Given $k_0, \Cshape,\Cpou>0$ and $\Cdegree\geq\Cdegreec>0$, 
 there exist 
 $\widetilde{C}_1, \widetilde{C}_2>0$ such that given $0<\widetilde{C}_3<\widetilde{C}_2$ there exists $\widetilde{C}_4>0$
 such that the following holds.
If Assumptions \ref{ass:fine}, \ref{ass:subdomain_def}, 
and \ref{ass:coarse_new} hold, 
$k\in K\cap [k_0,\infty)$,  $\Lambda\in\mathbb{Z}^+$, 
\beq\label{eq:degree}
kh \leq \widetilde{C}_1, \quad 
p_f \geq 1+\Cdegree \log (kL),
\eeq
\beq\label{eq:kh1new}
\frac{\widetilde{C}_3}{\Lambda^{1/2}}\leq k\delta, \quad  k\Hsub \leq \frac{\widetilde{C}_2}{\Lambda^{1/2}},
\eeq
\beq\label{eq:finalH}
\frac{k \Hcoarse}{p_c}\leq \frac{\widetilde{C}_4}{\Lambda^{1/2}}, 
\quad\tand\quad
p_c\geq 1 + \Cdegreec\log(kL),
\eeq
then the Galerkin solution in the fine space exists, is unique, and satisfies 
\beq\label{eq:QOfinal}
\N{u-u_h}_{H^1_k(\Omega)}\leq 2 \Ccont \min_{v_h \in \fdspace} \N{u-v_h}_{H^1_k(\Omega)},
\eeq
and the bounds in \eqref{eq:upper_add1} and \eqref{eq:upper_add2} (on, respectively, the norm and field of values of $Q$) hold.
\end{theorem}
}

\red{
\bre\label{rem:quantifiers}
The care regarding the quantifiers in Theorem \ref{thm:final} is needed since both $\Hsub\sim k^{-1}$ and $\delta \sim k^{-1}$, but $\delta\leq  \Hsub$ (by definition); therefore the precise constants in these $\sim$ relations matter (to avoid 
an impossible situation which would arise if $\delta$ were required to be $>\Hsub$).
The requirements in Theorem \ref{thm:final} on these constants are, in words, the following.
\bit
\item First that $k\Hsub$ and $kh$ must be sufficiently small
(the second inequality in \eqref{eq:kh1new} and the first inequality in \eqref{eq:degree}). 
\item Then, $k\delta$ is chosen so that the result holds for choices of $\delta\leq \Hsub$ (the first inequality in \eqref{eq:kh1new}). 
\item Then $k\Hcoarse/p_c$ is constrained to be sufficiently small depending on all the constants given so far (the first inequality in \eqref{eq:finalH}).
\eit
\ere
}
%

\red{
  
\bpf[Proof of Theorem \ref{thm:final}]
The fact that the Galerkin solution in the fine space exists, is unique, and satisfies \eqref{eq:QOfinal} follows from Corollary \ref{cor:hpFEM}, since $h$ and $p_f$ satisfying \eqref{eq:degree} also satisfy \eqref{eq:hpfem1}.

The rest of the result follows from Theorem \ref{thm:main_add}.
if we can show that
\ben
\item[(a)] Assumption \ref{ass:sesqui} (on the sesquilinear form) is satisfied,
\item[(b)] Assumption \ref{ass:coarse} (on the coarse-space Galerkin error) is satisfied,
\item[(c)] Assumption \ref{ass:subdomain} (boundedness of $Q_\ell$) is satisfied, and 
\item[(d)] the bound \eqref{eq:newHsub} holds (note that, due to the max, \eqref{eq:newHsub} is really two separate bounds -- one involving $\sigma$ and one involving $k\Hsub$).
\een
For (a), this follows immediately from the definition of the CAP sesquilinear form \eqref{eq:sesqui_CAP}. 

For (b), by
the Schatz argument (see, e.g., \cite[\S4]{MeSa:10}, \cite[Appendix B]{GS4})
 and Theorem \ref{thm:hpFEM}, 
if $\Hcoarse$ and $p_c$ satisfy \eqref{eq:finalH}, 
then the first bound in \eqref{eq:QO2a} is satisfied with $\Cnormo=2\Ccont$, and 
the second bound in \eqref{eq:QO2a} is satisfied with
\beq\label{e:sigmaApp}
\sigma\leq 2(\Ccont)^2 \mathcal{C}\Bigg[\frac{k\Hcoarse}{p_c}+
\Big(\frac{k\Hcoarse}{p_c}\Big)^Mk^{-N} (kL)^\polyexp +\frac{(\mathcal{C}\Hcoarse)^{p_c}}{k}
\max\Big\{\frac{k}{p_c+1},1\Big\}^{p_c+1}
(kL)^\polyexp\Bigg],
\eeq
where $\mathcal{C}, M,$ and $N$ are as in \eqref{e:etaUs}.

For (c), by the G\aa rding inequality \eqref{eq:Garding} and the Poincar\'e inequality \eqref{eq:Poincare} (applied with $D=\Omega_\ell$ and $L=\Hsub$), 
$a(\cdot,\cdot)$ is coercive on $H^1_0(\Omega_\ell)$ if $k\Hsub$ is sufficiently small, which follows from \eqref{eq:kh1new}, if necessary by reducing $\widetilde{C}_2$. Assumption \ref{ass:subdomain} then follows from Lemma \ref{lem:subdomain_coercive}.

For (d), we first concentrate on the bound involving $k\Hsub$ in \eqref{eq:newHsub}, and follow the steps in the bullet points in Remark \ref{rem:quantifiers}. Given $k_1, p_f>0$, fix $\widetilde{C}_1>0$ such that if $kh\leq \widetilde{C}_1$ (i.e., the first inequality in \eqref{eq:degree}) then the fine-space solution exists, is unique, and is $k$-uniformly quasi-optimal by Corollary \ref{cor:hpFEM}. Given $\widetilde{C}_1$ and $C_1>0$ from Theorem \ref{thm:main_add}, there exists $\widetilde{C}_2>0$ such that if $k\Hsub \leq \widetilde{C}_2/\Lambda^{1/2}$, then 
\beq\label{e:lastDay1}
(k\Hsub)^2 \Lambda (1+ \widetilde{C}_1)^2 \leq \frac{C_1^2}{2}.
\eeq
Then, given $0<\widetilde{C}_3\leq \widetilde{C}_2$, fix $k\delta \geq \widetilde{C}_3/\Lambda^{1/2}$ so that $(k\delta)^{-2} \leq \Lambda (\widetilde{C}_3)^{-2}$. Then there exists $\widetilde{C}_4>0$ such that if $k\Hcoarse/p_c \leq \widetilde{C}_4/\Lambda^{1/2}$ (i.e., the first inequality in \eqref{eq:finalH}) then 
\beq\label{e:lastDay2}
\frac{k\Hcoarse}{p_c} + \Lambda (\widetilde{C}_3)^{-2} \Big(\frac{k\Hcoarse}{p_c} \Big)^2 \leq 1.
\eeq
The combination of \eqref{e:lastDay1} and \eqref{e:lastDay2} imply that if $k\Hsub$, $k\delta$, $k\Hcoarse$, and $kh$ satisfy the bounds in \eqref{eq:degree}, \eqref{eq:kh1new}, and \eqref{eq:finalH}, then the bound on $k\Hsub$ in \eqref{eq:newHsub} holds.

It only remains to show that the bound involving  $\sigma$ in \eqref{eq:newHsub} holds. 
By the bound \eqref{e:sigmaApp}, given $\Cdegreec>0$, the bound involving  $\sigma$ in \eqref{eq:newHsub} can be achieved by making $k\Hcoarse/p_c$ sufficiently small. We therefore decrease $\widetilde{C}_4>0$ (if necessary) to achieve this, and the proof is complete.
\epf
}

\begin{appendix}

\section{The matrix form of the operator $\Qa$ \eqref{eq:pc}}\label{sec:appendix}

The fact that the matrix forms of $\Qa$ is given by \eqref{eq:matrixQ}
is a consequence of the following theorem.

\begin{theorem}\mythmname{\cite[Theorem 5.4]{GSV1}} 
\label{thm:repQ}
Let $v_h = \sum_{j\in \cJ_h} V_j \phi_j \ $  and $w_h = \sum_{j\in \cJ_h} W_j \phi_j \ $ be arbitrary elements of
$\fine$. Then, for $\ell = 0, \ldots, N$,
\begin{equation*}
  \big(Q_\ell v_h, w_h\big)_{H^1_k(\Omega)} =
  \big\langle
  {\matrixR_\ell}^\top
  {\matrixA}_\ell^{-1} {\matrixR}_\ell
  {\matrixA} \bV, \bW 
   \big\rangle_{\matrixD_k}. 
\end{equation*}
\end{theorem} 

\begin{proof}
By \eqref{eq:innerproducts}, it is sufficient to prove that, 
  for $\ell = 0, \ldots, N$,
  \begin{align} \label{iggapp2}
    Q_\ell v_h = \sum_{j \in \cJ_h} \big(\matrixR_\ell^\top \matrixA_\ell^{-1} \matrixR_\ell \matrixA \bV \big)_j \phi_j.
    \end{align}
    (Note that, by proving \eqref{iggapp2}, we correct typographical errors in the statement of the result 
    in  \cite[Theorem 5.4(ii)]{GSV1}.)

We first prove \eqref{iggapp2} for $\ell = 0$. On the one hand, since $\{ \Phi_q\}_{q \in \cJ_0} $ is a basis for $\cV_0$,  $Q_0v_h = \sum_{q \in \cJ_0} Z_q \Phi_q$,
     for some  coefficient vector $\bZ$. 
     By \eqref{eq:Phi1},  $Q_0v_h$ can then be written  in terms of the basis $\{\phi_j\}_{j \in \cJ_h}$ by 
     \begin{align} \label{iggapp3} Q_0 v_h = \sum_{q \in \cJ_0} Z_q \sum_{j \in \mathcal{J}_h}(\matrixR_0)_{qj} \phi_j = \sum_{j \in \mathcal{J}_h}  (\matrixR_0^\top \mathbf{Z})_j \phi_j   . \end{align}
On the other hand, the definition \eqref{eq:Q} of $Q_0$ and \eqref{eq:Phi1} imply that,
for all $p\in \cJ_0$, 
\begin{align}\label{eq:Friday1new}
  \sum_{q\in \cJ_0}  a(\Phi_q, \Phi_p)Z_q = a(Q_0 v_h , \Phi_p) = a(v_h , \Phi_p) =
  \sum_{j\in \cJ_h} V_j a(\phi_j,\Phi_p)= \sum_{j\in \cJ_h} \sum_{i\in \cJ_h}  (\matrixR_0)_{pi} a(\phi_j,\phi_i) V_j.
\end{align}
Recall that $a(\phi_i,\phi_j) = \matrixA_{i j}$ and  $ a(\Phi_q,\Phi_p) = (\matrixAzero)_{pq}$
(i.e., $\matrixAzero$ is the Galerkin matrix of $a(\cdot,\cdot)$
using the basis $\{\Phi_q\}_{q \in
\cJ_{0}}$ of $\cV_0$).
The expression 
 \eqref{eq:Friday1new} then becomes that $\matrixA_0 \bZ = \matrixR_0\matrixA \bV $; i.e., $\bZ = \matrixA_0^{-1} \matrixR_0\matrixA \bV $, and inserting this into \eqref{iggapp3} gives \eqref{iggapp2} for $\ell = 0$.

We now prove \eqref{iggapp2} for $\ell \in \{ 1, \ldots ,N\}$. Let $\cJ_h(\Omega_\ell)$ denote the index set for the freedoms of
a function in $\cV_\ell$. In analogue with the case $\ell = 0$, we write $Q_\ell v_h $ in two different ways:~first
\begin{align}
  \label{iggapp4}Q_\ell v_h = \sum_{j\in\cJ_h(\Omega_\ell)} Y_j \phi_j = \sum_{j\in\cJ_h} (\matrixR_\ell ^\top \bY)_j
  \phi_j, \end{align}
for some coefficient vector $\bY$. Then, by the definition  of $Q_\ell$, for all $i \in \cJ_h(\Omega_\ell)$, 
\begin{align}\label{iggapp5}
  \sum_{j\in \cJ(\Omega_\ell)}  a(\phi_j, \phi_i)Y_j = a(Q_\ell v_h , \phi_i) = a(v_h , \phi_i) =
  \sum_{j\in \cJ_h}   a(\phi_j,\phi_i)V_j.
\end{align}
It is straightforward to show that $\matrixA_\ell:= \matrixR_\ell \matrixA \matrixR_\ell^T$ is such that
$(\matrixA_\ell)_{i j} = a(\phi_j, \phi_i)$ for $i,j \in \cJ_h(\Omega_\ell)$. Therefore \eqref{iggapp5} implies that $\bY = \matrixA_\ell^{-1} \matrixR_\ell \matrixA \bV$, and inserting this into  \eqref{iggapp4} gives \eqref{iggapp2} for $\ell \in \{1, \ldots, N\}$. 
  \end{proof}

\section{Proofs of Theorems \ref{thm:nt_CAP} and \ref{thm:CAP}}\label{app:CAP}

The proofs of Theorems \ref{thm:nt_CAP} and \ref{thm:CAP} are small modifications of the proofs of  \cite[Lemma 4.5]{GGGLS2}
and \cite[Theorem A.2]{GGGLS2}, respectively, where we use Lemma \ref{lem:ibps} to deal with the boundary. 

The key difference is that the results of \cite{GGGLS2} assume that (i) both $\Omega$ and $\Omega_{\rm int}$ are hypercubes
and (ii) the coefficients of the PDE are constant in a neighbourhood of $\Omega$. 
Both of these assumptions are 
because \cite{GGGLS2} is focused on  the case when the radiation condition is approximated by a Cartesian PML, with (i) then necessary for the definition of a cartesian PML, and (ii) allowing one use a reflection argument near $\partial\Omega$ to avoid considering propagation of singularities up to the boundary (see \cite[Remark 2.5]{GGGLS2}).

Neither nor (i) and (ii) are needed for CAP. Indeed, CAP is defined for both $\Omega$ and $\Omega_{\rm int}$ bounded Lipschitz domains, avoiding (i), and the semiclassical principal symbol of the PDE is uniformly semiclassical elliptic near $\partial\Omega$ (in contrast to a cartesian PML, which is only semiclassically elliptic in the coordinate direction in which the scaling occurs), avoiding (ii). 

\bpf[Proof of Theorem \ref{thm:nt_CAP}]
The result of Theorem \ref{thm:nt_CAP} follows from \cite[Lemma 4.5]{GGGLS2} and Lemma \ref{lem:ibps} in the following way:~the contradiction argument 
in \cite[Lemma 4.5]{GGGLS2} is set up in exactly the same way; i.e., we obtain a sequence $v_n\in H^1_0(\Omega)$ such that $\| \opP v_n\|_{(H^1_k(\Omega))^*} \to 0$ as $n\to \infty$. To complete the proof, we need to show that $\|v_n\|_{H^1_k(\Omega)}\to 0$.

In constrast to the proof of \cite[Lemma 4.5]{GGGLS2}, we do not extend $v_n$ by reflection to an extended domain (denoted by $\widetilde{\Omega}$ in \cite{GGGLS2}), but instead work with an arbitrary
 $\chi\in C^{\infty}_{\rm comp}(\Omega)$ (instead of $\chi \in C^{\infty}_{\rm comp}(\widetilde\Omega)$).
 The propagation argument in \cite[Lemma 4.5]{GGGLS2} (which uses the ellipticity of the operator in the CAP region and the nontrapping assumption) then shows that 
both $\|\chi v_n\|_{L^2(\Omega)}$ and $\|\chi v_n\|_{H^1_k(\Omega)}\to 0$ as $n\to \infty$.

To complete the proof, we need to show that, for some $\chi\in C^{\infty}_{\rm comp}(\Omega)$,  $\|(1-\chi)v_n\|_{H^1_k(\Omega)}\to 0$ as $n\to\infty$ (to get that 
$\|v_n\|_{H^1_k(\Omega)}\to 0$). 
Choose $\chi \in C^{\infty}_{\rm comp}(\Omega; [0,1])$ such that $\supp(1-\chi)\subset \{V\geq a>0\}$ and choose $\phi\in C^\infty(\Rea^d)$ such that $\phi= 1$ on $\supp(1-\chi)$ and $\supp \,\phi \subset \{V\geq a>0\}$; observe that such a $\phi$ satisfies the assumptions of Lemma \ref{lem:ibps}. Therefore, by the support properties of $1-\chi$ and $\phi$, and Lemma \ref{lem:ibps}, 
\begin{align*}
\N{(1-\chi)v_n}_{H^1_k(\Omega)}\leq C\N{\phi v_n}_{H^1_k(\Omega)}&\leq C \N{\opP(\phi v_n)}_{(H^1_k(\Omega))^*}\\
&\leq C \Big(
\N{\phi\opP v_n}_{(H^1_k(\Omega))^*} + \N{ (\opP\phi-\phi\opP)v_n}_{(H^1_k(\Omega))^*} 
\Big)\\
&\leq C' \Big( \N{\opP v_n}_{(H^1_k(\Omega))^*} + \N{\widetilde{\chi} v_n}_{L^2(\Omega)}\Big)
\end{align*}
for some $\widetilde{\chi}\in C^\infty_{\rm comp}(\Omega)$ with $\widetilde\chi= 1$ on $\supp \nabla\phi$. 
The propagation argument from the proof of  \cite[Lemma 4.5]{GGGLS2} shows that 
 $\|\widetilde{\chi} v_n\|_{L^2(\Omega)} \to 0$, 
and $\| \opP v_n\|_{(H^1_k(\Omega))^*} \to 0$ by construction. Therefore $\|(1-\chi)v_n\|_{H^1_k(\Omega)}\to 0$ as $n\to\infty$, and the result follows.
\epf

\bpf[Proof of Theorem \ref{thm:CAP}]
The proof of Theorem \ref{thm:CAP} is then exactly the same as the proof of \cite[Theorem A.2]{GGGLS2} (with Theorem \ref{thm:nt_CAP} used in place of \cite[Lemma 4.5]{GGGLS2}), noting that the propagation result of \cite[Lemma 4.1]{GGGLS2} is only used on compact subsets of $\Omega$ (i.e., there is no propagation up to the boundary), since 
the norms on the left-hand side of \eqref{eq:small} all involve $\chi \in C^\infty_{\rm comp}(\Omega_{\rm int})$.
\epf

\section{Discussion of the  numerical experiments in \cite{BoDoJoTo:21}}\label{sec:PHT}

\red{In \S\ref{sec:discussion} we stated that the the experiments in \cite{BoDoJoTo:21}, which consider $p_c=p_f=2$ and a hybrid Schwarz preconditioner, show that the number of GMRES iterations 
(i) grows slowly with $k$ 
when the number of degrees of freedom per subdomain is kept constant, and 
(ii) grows with $k$ if the coarse space does not resolve the oscillatory/propagative nature of the solution.
We now give more detail on these two points.} The \emph{grid coarse space method} of \cite{BoDoJoTo:21} involves FEM discretisations with $p_c=p_f=2$, 
10 points per wavelength in the fine space, and 5 points per wavelength in the coarse space (i.e., both $h$ and $\Hcoarse \sim k^{-1}$)
 and GMRES is then applied with a hybrid Schwarz preconditioner with impedance boundary conditions on the subdomains and minimal overlap.
When $k$ is doubled and the number of subdomains ($N$) increases by $2^d$ (so that the number of degrees of freedom per subdomain is kept constant -- i.e., close to the set up in Theorems \ref{thm:informal_pwp_add}), the number of iterations 
goes from $41$ ($f=10$, $N=40$, where $f$ is frequency) to $44$ ($f=20$, $N=160$) in \cite[Table 1]{BoDoJoTo:21} for the 2-d Marmousi model
and from $11\, (k=100, N=20)$ to $16\, (k=200, N=160)$ \cite[Table 7]{BoDoJoTo:21} for the 3-d cobra cavity.
Finally, as mentioned above, \cite[Tables 5 and 9]{BoDoJoTo:21} shows that the number of iterations is large if 
 there are only 5 points per wavelength in the fine space, and 2.5 points per wavelength in the coarse space.

\section*{Acknowledgements}

The authors thank Martin Averseng (CNRS, Angers), David Lafontaine (CNRS Toulouse), Guanglian Li (University of Hong Kong), Chupeng Ma (Green Bay University),  Daniel Peterseim (Universit\"at Augsburg),  
Pierre-Henri Tournier (CNRS Paris), and particularly Th\'eophile Chaumont-Frelet (INRIA, Lille) and Jeffrey Galkowski (University College London) 
for useful discussions. 
\red{The authors also thank the referees and associate editor for many useful comments.}
The authors acknowledge the hospitality of the Tsinghua Sanya International Mathematics Forum at the workshop ``Advanced solvers for frequency-domain wave problems and applications" in January 2025, where key parts of this research were conducted.

\end{appendix}

\footnotesize{
\bibliographystyle{plain}
\bibliography{combined.bib}
}

\end{document}